\documentclass[preprint,12pt]{elsarticle}
\usepackage[T1]{fontenc}
\usepackage[utf8]{inputenc}
\usepackage{mathrsfs,amsmath,amssymb,amsthm,mathtools,bbm}
\usepackage{enumitem}
\usepackage{xcolor,graphicx}
\usepackage[a4paper,margin=1in]{geometry}
\definecolor{Bittersweet}{HTML}{C04F17}
\usepackage[unicode,bookmarksnumbered,colorlinks,plainpages,linktocpage]{hyperref}
\definecolor{Darkgreen}{rgb}{0,0.4,0}
\hypersetup{%
    pdfborder = {0 0 0},
    colorlinks,
    citecolor=blue,
    filecolor=Darkgreen,
    linkcolor=Bittersweet,
    urlcolor=blue}

\usepackage{comment}

\newtheorem{theorem}{Theorem}
\newtheorem{proposition}[theorem]{Proposition}
\newtheorem{corollary}[theorem]{Corollary}
\newtheorem{lemma}[theorem]{Lemma}

\theoremstyle{definition}
\newtheorem{definition}[theorem]{Definition}
\newtheorem{remark}[theorem]{Remark}
\newtheorem{example}[theorem]{Example}

\newcommand{\Sym}{\operatorname{Sym}}
\newcommand{\Hess}{\operatorname{Hess}}
\newcommand{\wt}{\operatorname{wt}}

\begin{document}
\hyphenpenalty=10000
\exhyphenpenalty=10000
\sloppy
%%%%%%%%%%%%%%%%%%%%%%%%%%%%%%%%%%%%%%%%%%%%%%%%%
%%%%%%%%%%%%%%%%%%%%%%%%%%%%%%%%%%%%%%%%%%%%%%%%%
\begin{frontmatter}
%%%%%%%%%%%%%%%%%%%%%%%%%%%%%%%%%%%%%%%%%%%%%%%%%
\title{On the Kronecker Products of Symmetric Persistent Tensors}
%%%%%%%%%%%%%%%%%%%%%%%%%%%%%%%%%%%%%%%%%%%%%%%%%
%\author{Masoud Gharahi\corref{cor1}}
\author{Masoud Gharahi}
\ead{masoud.gharahi@gmail.com; masoud.gharahighahi@uj.edu.pl}
\address{Faculty of Physics, Astronomy and Applied Computer Science, Jagiellonian University, 30-348 Kraków, Poland}
%\author[uj]{Masoud Gharahi\corref{cor1}}
%\ead{masoud.gharahighahi@uj.edu.pl}
%\address[uj]{Faculty of Physics, Astronomy and Applied Computer Science, Jagiellonian University, 30-348 Kraków, Poland}
%%%%%%%%%%%%%%%%%%%%%%%%%%%%%%%%%%%%%%%%%%%%%%%%%
%%%%%%%%%%%%%%%%%%%%%%%%%%%%%%%%%%%%%%%%%%%%%%%%%
\begin{abstract}
Persistent tensors form a recursively defined class adapted to the substitution method and provide nontrivial lower bounds on tensor rank. We study the behavior of symmetric persistent tensors under Kronecker products. We establish a global differentiation identity expressing the Hessian matrix of a Kronecker product of homogeneous polynomials in terms of the Hessian matrices of its factors; although no corresponding determinant identity holds globally, the Hessian polynomials factor exactly at decomposable points. This yields a multiplicative formula for the distinguished Hessian coefficients of isobaric forms and closure under Kronecker products for symmetric persistent tensors that are isobaric of the distinguished weight, including iterated products and powers. Combined with the classification in small dimensions, this implies closure when both factors have dimension at most three, and for persistent cubics when both have dimension at most four. We further give sufficient closure criteria via simultaneous strict triangularizability of normalized Hessian spaces, for cubics and then arbitrary degree, and show that the distinguished isobaric class satisfies this condition. Finally, we show that symmetric persistence is not preserved in general by constructing a persistent cubic $f\in\operatorname{Sym}^3\mathbbm{C}^{12}$ such that $f\boxtimes f\in\operatorname{Sym}^3\mathbbm{C}^{144}$ is not persistent. The obstruction occurs at a nondecomposable point, showing that the Hessian factorization on the Segre variety does not extend to the full tensor product space.
\end{abstract}
%%%%%%%%%%%%%%%%%%%%%%%%%%%%%%%%%%%%%%%%%%%%%%%%%
%%%%%%%%%%%%%%%%%%%%%%%%%%%%%%%%%%%%%%%%%%%%%%%%%
%%%%%%%%%%%%%%%%%%%%%%%%%%%%%%%%%%%%%%%%%%%%%%%%%
\end{frontmatter}
%%%%%%%%%%%%%%%%%%%%%%%%%%%%%%%%%%%%%%%%%%%%%%%%%
\begin{center}
\emph{Dedicated to Giorgio Ottaviani}
\end{center}
%%%%%%%%%%%%%%%%%%%%%%%%%%%%%%%%%%%%%%%%%%%%%%%%%
%%%%%%%%%%%%%%%%%%%%%%%%%%%%%%%%%%%%%%%%%%%%%%%%%
\section{Introduction}

Persistent tensors were introduced in \cite{GL24,Gharahi} as a recursively defined class of tensors for which repeated substitution yields nontrivial lower bounds on tensor rank. For a bipartite tensor, persistence coincides with conciseness in the first factor. For tensors of order at least three, persistence requires first-factor conciseness together with persistence of all contractions outside a proper linear subspace.

It was conjectured in \cite{GL24,Gharahi} that the Kronecker product of persistent tensors is again persistent. Shitov constructed a counterexample to this unrestricted conjecture \cite{Shitov}. Since his example is nonsymmetric, it is natural to ask to what extent persistence is preserved under Kronecker products in the symmetric setting. This motivates the search for structured symmetric subclasses on which persistence is preserved under Kronecker products.

The symmetric setting is especially natural. Symmetric tensors of order $n$ on a vector space $V$ are naturally identified with homogeneous degree-$n$ polynomials in $\Sym^nV$. Recent work established a complete Hessian characterization of symmetric persistent tensors \cite{GO25}. More precisely, persistence is equivalent to a perfect-power condition on the determinants of the partially polarized Hessians \cite[Theorem~2]{GO25}. This characterization connects persistence with the classical Hessian construction from algebraic geometry \cite{Dolgachev}.

The starting point of the present paper is a global differentiation identity for Kronecker products of arbitrary homogeneous polynomials. More precisely, for arbitrary $f\in\Sym^n\mathbbm{C}^{d_1}$ and $g\in\Sym^n\mathbbm{C}^{d_2}$, we prove that the Hessian matrices satisfy
\begin{equation}\label{intro-global-Hessian}
\mathcal{H}_{f\boxtimes g}=\frac{1}{n(n-1)}\,\mathcal{H}_f\boxtimes\mathcal{H}_g,
\end{equation}
where the Kronecker product on the right-hand side is understood entrywise in the sense of the polynomial Kronecker product defined below. This is a global identity on the full tensor-product space and holds without any persistence assumption.

A crucial point is that, even for symmetric persistent tensors, taking determinants in \eqref{intro-global-Hessian} does not, in general, yield the corresponding identity for the Hessian polynomials. After evaluation at a decomposable point $v\otimes w$, however, the right-hand side becomes the ordinary matrix Kronecker product of the evaluated Hessian matrices, and the usual determinant formula applies. We obtain
\begin{equation}
\Hess(f\boxtimes g)(v\otimes w)=\left(\frac{1}{n(n-1)}\right)^{\!d_1d_2}\Hess(f)(v)^{d_2}\Hess(g)(w)^{d_1}.
\end{equation}
Thus the Hessian determinant admits an exact factorization on the Segre variety.

The identity at decomposable points is the main algebraic ingredient in our positive results. We first apply it to the distinguished isobaric weight spaces. For persistent forms in these spaces, the perfect-power Hessian certificate contains a distinguished coefficient which detects the relevant monomial structure. We prove that this coefficient is multiplicative under Kronecker products. Together with the preservation of isobaricity, this yields closure of the corresponding class of symmetric persistent tensors under Kronecker products. An explicit example is illustrated in Figure~\ref{fig:Kronecker-Persistent}. As a consequence, this class is also closed under iterated Kronecker products and Kronecker powers.
\begin{figure}[ht]
\centering
\includegraphics[width=\textwidth]{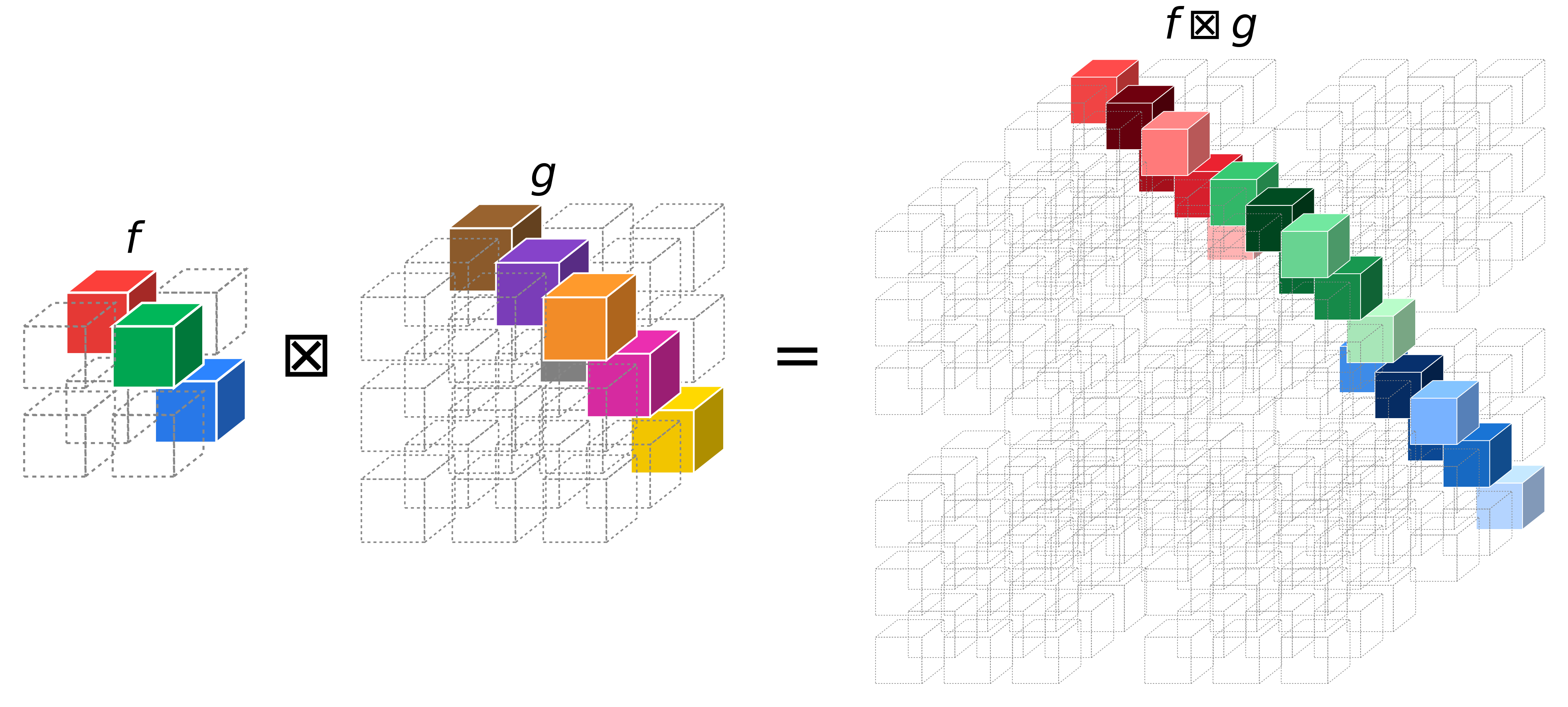}
\caption{Schematic representation of the Kronecker product of the symmetric persistent tensors $f=x_0^2x_1\in\Sym^3\mathbbm{C}^2$ and $g=y_0^2y_2+y_0y_1^2\in\Sym^3\mathbbm{C}^3$. The colored cubes represent nonzero entries, while the dotted cubes represent zero entries. The Kronecker product $f\boxtimes g\in\Sym^3(\mathbbm{C}^2\otimes\mathbbm{C}^3)\simeq\Sym^3\mathbbm{C}^6$ is represented as a $2\times2\times2$ block tensor, each block being a $3\times3\times3$ array. The three nonzero entries of $f$ produce three nonzero blocks, each reproducing the support pattern of $g$, as indicated by the red-, green-, and blue-toned colors. See Example~\ref{example-isobaric}.}
\label{fig:Kronecker-Persistent}
\end{figure}

The classification theorem of \cite{GO25} then gives closure beyond a fixed choice of isobaric coordinates. In particular, we obtain closure when both factors have dimension at most three, and closure for cubic persistent tensors when both factors have dimension at most four. We also give a complementary matrix-theoretic sufficient condition for Kronecker closure in terms of simultaneous strict triangularizability of normalized Hessian spaces, extend this criterion from cubic forms to arbitrary degree, and show that the distinguished isobaric class is contained in this triangularizable class.

The limitation of the decomposable-point argument is essential. Decomposable tensors form the Segre variety
\begin{equation}
\operatorname{Seg}\big(\mathbbm{P}(V)\times\mathbbm{P}(W)\big)\subseteq\mathbbm{P}(V\otimes W),
\end{equation}
which is a proper subvariety whenever $\dim V,\dim W>1$. Consequently, an identity verified on decomposable points need not extend to the whole tensor-product space: the difference of two candidate polynomial identities may be a nonzero element of the Segre ideal while vanishing identically on every decomposable tensor.

We show that this obstruction actually occurs for symmetric persistent tensors. We construct an explicit persistent cubic $f\in\Sym^3\mathbbm{C}^{12}$ for which $f\boxtimes f\in\Sym^3(\mathbbm{C}^{12}\otimes\mathbbm{C}^{12})\simeq\Sym^3\mathbbm{C}^{144}$ is not persistent. Motivated by the work of Meisters and Olech \cite{MO91}, the construction starts from a quadratic polynomial map $Q:\mathbbm{C}^5\to\mathbbm{C}^5$ whose Jacobian matrix is nilpotent at every point. From $Q$ we construct a cubic $f$ satisfying $\Hess(f)=(2s)^{12}$, so the cubic Hessian characterization of \cite{GO25} implies that $f$ is persistent. We then exhibit a nondecomposable tensor $z=u_0\otimes u_0+u_1\otimes u_1$ for which the normalized Hessian operator associated with $f\boxtimes f$ is not nilpotent. This contradicts the perfect-power Hessian condition that persistence of $f\boxtimes f$ would impose.

The paper is organized as follows. In Section~2 we recall partial polarization, symmetric persistence, the Hessian characterization, Jacobian and Hessian matrices, and isobaric polynomials. In Section~3 we introduce Kronecker products and establish the global differentiation identity and the Hessian-matrix identity. We then derive the Hessian factorization at decomposable points and prove preservation of isobaricity and multiplicativity of the distinguished Hessian coefficient, obtaining closure under Kronecker products, iterated products, and Kronecker powers in the distinguished isobaric class. Section~4 treats the classified low-dimensional cases. In Section~5 we develop a complementary matrix-space viewpoint based on normalized Hessian spaces and establish triangularizability criteria for Kronecker closure, first for cubic forms and then in arbitrary degree. Finally, in Section~6 we construct a counterexample showing that symmetric persistent tensors are not closed under Kronecker products in general.

%%%%%%%%%%%%%%%%%%%%%%%%%%%%%%%%%%%%%%%%%%%%%%%%%%%%%%%%%%%%
%%%%%%%%%%%%%%%%%%%%%%%%%%%%%%%%%%%%%%%%%%%%%%%%%%%%%%%%%%%%
\section{Preliminaries}

Throughout the paper, all vector spaces are finite-dimensional over $\mathbbm{C}$. Let $V$ be a vector space of dimension $d$. After choosing a basis $e_0,\ldots,e_{d-1}$ of $V$, a symmetric tensor $f\in\Sym^n V$ may equivalently be represented as a homogeneous polynomial of degree $n$ in the coordinate variables $x_0,\ldots,x_{d-1}$. We use the same symbol $f$ for the symmetric tensor and its associated homogeneous polynomial whenever no confusion can arise. After fixing the basis, we use the induced coordinate identification when writing contractions of symmetric tensors as directional derivatives of their associated polynomials.

Let $T_f$ denote the associated symmetric $n$-linear form. We use the restitution convention
\begin{equation}\label{restitution-convention}
f(v)=T_f(v,\ldots,v), \qquad v\in V.
\end{equation}
For $v=\sum_{i=0}^{d-1}v_i e_i$, recall that differentiation in the direction $v$ is $\partial_v=\sum_{i=0}^{d-1}v_i\frac{\partial}{\partial x_i}$. Hence, for arbitrary vectors $v^{(1)},\ldots,v^{(r)}\in V$ and $0\leq r\leq n$, the $r$-fold partial polarization of $f$ is
\begin{equation}\label{partial-polarization}
\partial_{v^{(1)}}\cdots\partial_{v^{(r)}}f(x)\!=
\!\!\!\sum_{i_1,\ldots,i_r=0}^{d-1}\!\!v^{(1)}_{i_1}\cdots v^{(r)}_{i_r}\frac{\partial^{\,r}f}{\partial x_{i_1}\cdots\partial x_{i_r}}(x)\!=\!\frac{n!}{(n-r)!}\,T_f(v^{(1)},\ldots,v^{(r)},x,\ldots,x),
\end{equation}
where the last $n-r$ arguments of $T_f$ are equal to $x$. For $r=n-2$, this partial polarization is a quadratic polynomial in $x$ whose coefficients depend multilinearly on $v^{(1)},\ldots,v^{(r)}$. For brevity, we denote it by
\begin{equation}\label{notation-partial-polarization}
f_{v^{(1)},\ldots,v^{(r)}}(x)
\coloneqq\partial_{v^{(1)}}\cdots\partial_{v^{(r)}}f(x).
\end{equation}

\begin{definition}[Hessian]
For a homogeneous polynomial $f\in\Sym^n V$, we denote by
\begin{equation}
\mathcal{H}_f(x)\coloneqq\left(\frac{\partial^{2} f}{\partial x_i\,\partial x_j}(x)\right)_{0\leq i,j\leq d-1}
\end{equation}
its \emph{Hessian matrix} at the point
$x=(x_0,\ldots,x_{d-1})$, and by
\begin{equation}
\Hess(f)\coloneqq\det\mathcal{H}_f(x)\in\mathbbm{C}[x_0,\ldots,x_{d-1}]
\end{equation}
its \emph{Hessian polynomial}. Henceforth, we shall refer to the polynomial $\Hess(f)$ simply as the \emph{Hessian} of $f$.
\end{definition}

For each $v\in V$, the Hessian matrix $\mathcal{H}_f(v)$ defines a symmetric bilinear form on $V$, also denoted by $\mathcal{H}_f(v)$, via
\begin{equation}
\mathcal{H}_f(v)(a,c)\coloneqq\sum_{i,j=0}^{d-1}a_i c_j
\frac{\partial^{2}f}{\partial x_i\,\partial x_j}(v).
\end{equation}
where $a=\sum_{i=0}^{d-1}a_i e_i$ and $c=\sum_{j=0}^{d-1}c_j e_j$. With the restitution convention \eqref{restitution-convention}, one has
\begin{equation}\label{hessian-polarization}
\mathcal{H}_f(v)(a,c)=n(n-1)T_f(a,c,v,\ldots,v).
\end{equation}

\begin{definition}[Jacobian]
Let $F=(f_0,\ldots,f_{d-1}):V\to V$ be a polynomial map. With respect to coordinates $x=(x_0,\ldots,x_{d-1})$ on $V$, its Jacobian matrix at $x$ is
\begin{equation}
\mathcal{J}_F(x)\coloneqq\left(\frac{\partial f_i}{\partial x_j}(x)\right)_{0\leq i,j\leq d-1}.
\end{equation}
\end{definition}

\begin{remark}
For a homogeneous polynomial $f\in\Sym^nV$, its gradient map is
\begin{equation}
\nabla f=\Big(\frac{\partial f}{\partial x_0},\ldots,\frac{\partial f}{\partial x_{d-1}}\Big).
\end{equation}
The Jacobian matrix of the gradient map is precisely the Hessian matrix of $f$:
\begin{equation}
\mathcal{J}_{\nabla f}(x)=\left(\frac{\partial^{2}f}{\partial x_i\partial x_j}(x)\right)_{0\leq i,j\leq d-1}=\mathcal{H}_f(x).
\end{equation}
\end{remark}

\begin{definition}
A symmetric tensor $f\in\Sym^nV$ is called \emph{concise} if its first partial derivatives are linearly independent. Equivalently, there is no proper subspace $W\subsetneq V$ such that $f\in\Sym^n W$. Equivalently, there is no invertible linear change of coordinates under which $f$ can be expressed
as a polynomial in fewer than $\dim V$ variables.
\end{definition}

\begin{definition}
A symmetric tensor $f\in\Sym^2V$ is called \emph{persistent} if the corresponding quadratic form is nonsingular. For $n\ge3$, a symmetric tensor $f\in\Sym^nV$ is called \emph{persistent} if it is concise and there exists a hyperplane $S\subsetneq V^\vee$ such that the contraction $\partial_u f\in\Sym^{n-1}V$ is persistent for every $u\in V^\vee\setminus S$.
\end{definition}

Here $\partial_u f$ denotes the contraction of the symmetric tensor $f$ with the covector $u\in V^\vee$.

We shall use the following Hessian characterization established in~\cite{GO25}.
\begin{theorem}[Hessian characterization {\cite[Theorem~2 and Corollary~4]{GO25}}]\label{thm:hessian}
Let $f\in\Sym^n\mathbbm{C}^d$, where $n\geq2$.
\begin{enumerate}[label=\textup{(\roman*)}]
\item If there exist a scalar $\lambda\in\mathbbm{C}^{\times}$ and a nonzero linear form $\ell$ such that
\begin{equation}
\Hess(f)=\lambda\,\ell^{d(n-2)},
\end{equation}
then $f$ is persistent.

\item If $n=3$, then the converse in \textnormal{(i)} also holds. More precisely, a cubic form $f\in\Sym^3\mathbbm{C}^d$ is persistent if and only if there exist $\lambda\in\mathbbm{C}^{\times}$ and a nonzero linear form $\ell$ such that
\begin{equation}
\Hess(f)=\lambda\,\ell^d.
\end{equation}

\item The tensor $f$ is persistent if and only if there exists a nonzero multihomogeneous polynomial $P_f(v^{(1)},\ldots,v^{(n-2)})$ of multidegree $(1,\ldots,1)$ such that
\begin{equation}\label{hessian-characterization}
\Hess\big(f_{v^{(1)},\ldots,v^{(n-2)}}(x)\big)=
\left[P_f\big(v^{(1)},\ldots,v^{(n-2)}\big)\right]^d,
\end{equation}
where the Hessian is taken with respect to $x=(x_0,\ldots,x_{d-1})$.

\item If $f$ is persistent, then there exists a nonzero homogeneous polynomial
$g\in\Sym^{n-2}\mathbbm{C}^d$ such that
\begin{equation}\label{hessian-perfect-power}
\Hess(f)=g^d.
\end{equation}
In particular, the Hessian of every persistent symmetric tensor is a perfect $d$-th power.
\end{enumerate}
\end{theorem}

We also recall the notion of an isobaric polynomial.
\begin{definition}
Let $x_0,\ldots,x_{d-1}$ be coordinates on $\mathbbm{C}^d$, and assign the weight
\begin{equation}
\wt(x_j)\coloneqq j,\qquad j=0,\ldots,d-1.
\end{equation}
The weight of a monomial $\mathbf{x}^{\alpha}=x_0^{\alpha_0}\cdots x_{d-1}^{\alpha_{d-1}}$
is $\wt(\mathbf{x}^{\alpha})=\sum_{j=0}^{d-1}j\alpha_j$. A homogeneous polynomial is called \emph{isobaric of weight $\delta$} if every monomial in its support has weight $\delta$.
\end{definition}

The following result is proved in~\cite{GO25}.
\begin{lemma}[{\cite[Lemma~15]{GO25}}]\label{lem:Hessian}
Let $f\in\Sym^n\mathbbm{C}^d$ be isobaric of weight $d-1$. Then
\begin{equation}
\Hess(f)=\lambda\,x_0^{d(n-2)}
\end{equation}
for some uniquely determined scalar $\lambda\in\mathbbm{C}$.
\end{lemma}
In particular, if $\lambda\neq 0$, then $f$ is persistent.

\paragraph{Notation}
Throughout the paper, the symbol $\otimes$ denotes the tensor product, whereas $\boxtimes$ denotes the Kronecker product. In particular, if $A$ and $B$ are bilinear forms or endomorphisms, then the matrix representing $A\boxtimes B$ with respect to the corresponding product bases is the usual matrix Kronecker product.

%%%%%%%%%%%%%%%%%%%%%%%%%%%%%%%%%%%%%%%%%%%%%%%%%%%%%%%%%%%%
%%%%%%%%%%%%%%%%%%%%%%%%%%%%%%%%%%%%%%%%%%%%%%%%%%%%%%%%%%%%
\section{Kronecker products}

Let $n\geq2$. Let $V\simeq\mathbbm{C}^{d_1}$ and $W\simeq\mathbbm{C}^{d_2}$, and let $f\in\Sym^n V$ and $g\in\Sym^n W$ be symmetric tensors. By the restitution convention \eqref{restitution-convention},
\begin{equation}\label{restitution-convention-f-g}
f(v)=T_f(v,\ldots,v), \qquad g(w)=T_g(w,\ldots,w).
\end{equation}

We now define the Kronecker product in terms of the associated symmetric multilinear forms.

\begin{definition}\label{def:KroneckerProduct}
Let $T_f$ and $T_g$ denote the symmetric $n$-linear forms associated with $f\in\Sym^n V$ and $g\in\Sym^n W$, respectively. The \emph{Kronecker product} $f\boxtimes g\in\Sym^n(V\otimes W)$ is the symmetric tensor whose associated $n$-linear form is the composition
\begin{equation}
(V\otimes W)^{\otimes n} \longrightarrow V^{\otimes n}\otimes W^{\otimes n} \xrightarrow{\,\,T_f\otimes T_g\,\,}\mathbbm{C},
\end{equation}
where the first arrow is the canonical shuffle isomorphism $(v_1\otimes w_1)\otimes\cdots\otimes(v_n\otimes w_n) \longmapsto (v_1\otimes\cdots\otimes v_n)\otimes(w_1\otimes\cdots\otimes w_n)$. Equivalently,
\begin{equation}\label{Kronecker-definition}
T_{f\boxtimes g}(v_1\otimes w_1,\ldots,v_n\otimes w_n)=T_f(v_1,\ldots,v_n)\,T_g(w_1,\ldots,w_n)
\end{equation}
for all $v_i\in V$ and $w_i\in W$.
\end{definition}
The canonical shuffle construction shows that $T_{f\boxtimes g}$ is well defined, and its symmetry follows immediately from the symmetry of $T_f$ and $T_g$.

Choose coordinates $x_0,\ldots,x_{d_1-1}$ and $y_0,\ldots,y_{d_2-1}$ on $V$ and $W$, respectively. Let $z_{ij}$ denote the coordinate corresponding to the basis vector $e_i\otimes\varepsilon_j$ of $V\otimes W$.

\begin{remark}\label{rem:coordinate-Kronecker}
With respect to the coordinates introduced above, Definition~\ref{def:KroneckerProduct} is equivalent to the explicit formula
\begin{equation}\label{coordinate-Kronecker}
f\boxtimes g=\frac{1}{(n!)^2}
\sum_{i_1,\ldots,i_n=0}^{d_1-1}
\sum_{j_1,\ldots,j_n=0}^{d_2-1}
\frac{\partial^{\,n}f}{\partial x_{i_1}\cdots\partial x_{i_n}}
\frac{\partial^{\,n}g}{\partial y_{j_1}\cdots\partial y_{j_n}}
z_{i_1j_1}\cdots z_{i_nj_n}.
\end{equation}
\end{remark}

\begin{lemma}[Kronecker products preserve isobaricity]
\label{lem:Kronecker-isobaric}
Assume that $f$ is isobaric of weight $d_1-1$ and that $g$ is
isobaric of weight $d_2-1$. Order the basis of $V\otimes W$
lexicographically and assign $\wt(z_{ij})=d_2i+j$. Then $f\boxtimes g$ is isobaric of weight $d_1d_2-1$.
\end{lemma}
\begin{proof}
With respect to the chosen bases and the restitution convention \eqref{restitution-convention}, let $f_{i_1\cdots i_n}$ and $g_{j_1\cdots j_n}$ denote the symmetric coordinate arrays of $T_f$ and $T_g$. Accordingly,
\begin{equation*}
f=\sum_{i_1,\ldots,i_n}f_{i_1\cdots i_n}\,x_{i_1}\cdots x_{i_n},
\qquad\qquad
g=\sum_{j_1,\ldots,j_n}g_{j_1\cdots j_n}\,y_{j_1}\cdots y_{j_n}.
\end{equation*}
By the definition of the Kronecker product,
\begin{equation*}
f\boxtimes g=\sum_{\substack{i_1,\ldots,i_n\\ j_1,\ldots,j_n}}
f_{i_1\cdots i_n}\,g_{j_1\cdots j_n}\,z_{i_1j_1}\cdots z_{i_nj_n}.
\end{equation*}
Every summand with nonzero coefficient in the displayed tensor-array expansion has the form $f_{i_1\cdots i_n}g_{j_1\cdots j_n}z_{i_1j_1}\cdots z_{i_nj_n}$. Since $f_{i_1\cdots i_n}g_{j_1\cdots j_n}\neq0$ over $\mathbbm{C}$, both $f_{i_1\cdots i_n}$ and $g_{j_1\cdots j_n}$ are nonzero. Hence, by the isobaricity of $f$ and $g$,
\begin{equation*}
\sum_{r=1}^n i_r=d_1-1,
\qquad\qquad
\sum_{r=1}^n j_r=d_2-1.
\end{equation*}
Therefore,
\begin{equation*}
\wt(z_{i_1j_1}\cdots z_{i_nj_n})=\sum_{r=1}^n\wt(z_{i_rj_r})= d_2\sum_{r=1}^n i_r+\sum_{r=1}^n j_r=d_2(d_1-1)+(d_2-1)=d_1d_2-1.
\end{equation*}
Hence every summand in the tensor-array expansion has weight $d_1d_2-1$. After collecting equal monomials, every monomial with nonzero coefficient still has this weight. Therefore $f\boxtimes g$ is isobaric of weight $d_1d_2-1$.
\end{proof}

We shall use the following basic determination principle for Segre--Veronese tensors, which follows from the standard polarization identity and the resulting fact that pure powers span the corresponding symmetric powers; see, for instance, \cite{IK,Landsberg}.

\begin{lemma}[Segre--Veronese determination]
\label{lem:Segre-Veronese-determination}
Let $V$ and $W$ be finite-dimensional complex vector spaces, and let $A,B\in\bigotimes_{i=1}^{r} \big(\Sym^{n_i}V^\vee\otimes\Sym^{m_i}W^\vee\big)$. If
\begin{equation*}
A\big((v^{(1)})^{n_1}\otimes(w^{(1)})^{m_1},
\ldots,(v^{(r)})^{n_r}\otimes(w^{(r)})^{m_r}\big)\!=
\!B\big((v^{(1)})^{n_1}\otimes(w^{(1)})^{m_1},
\ldots,(v^{(r)})^{n_r}\otimes(w^{(r)})^{m_r}\big)
\end{equation*}
for every $v^{(i)}\in V$ and $w^{(i)}\in W$, $1\leq i\leq r$, then $A=B$.
\end{lemma}
\begin{proof}
For each $1\leq i\leq r$, the pure powers $(v^{(i)})^{n_i}$, with $v^{(i)}\in V$, and $(w^{(i)})^{m_i}$, with $w^{(i)}\in W$, span $\Sym^{n_i}V$ and $\Sym^{m_i}W$, respectively. Indeed, this follows from the fact that the Veronese variety is linearly nondegenerate, and therefore its affine cone spans the corresponding symmetric power. Consequently, the Segre--Veronese variety is also linearly nondegenerate, and hence
\begin{equation*}
\operatorname{Span}
\left\{(v^{(i)})^{n_i}\otimes(w^{(i)})^{m_i}:v^{(i)}\in V,\ w^{(i)}\in W\right\}=\Sym^{n_i}V\otimes\Sym^{m_i}W.
\end{equation*}
Taking tensor products over $i=1,\ldots,r$, we obtain
\begin{equation*}
\operatorname{Span}\Big\{\bigotimes_{i=1}^{r} \big((v^{(i)})^{n_i}\otimes(w^{(i)})^{m_i}\big):v^{(i)}\in V,\ w^{(i)}\in W\Big\}=\bigotimes_{i=1}^{r} \big(\Sym^{n_i}V\otimes\Sym^{m_i}W\big).
\end{equation*}
Since $A-B$ is a linear functional on this tensor product and vanishes on a spanning set, it vanishes identically. Hence $A=B$.
\end{proof}

The Kronecker product is compatible with differentiation in the sense of global polynomial identities. The following proposition expresses each partial derivative of a Kronecker product as the Kronecker product of the corresponding partial derivatives of its factors.

\begin{proposition}[Differentiation of Kronecker products]\label{prop:Kronecker-differentiation}
Let $f\in\Sym^nV$ and $g\in\Sym^nW$. Then, for every $0\leq i\leq d_1-1$ and $0\leq j\leq d_2-1$,
\begin{equation}\label{first-derivative-Kronecker}
\frac{\partial(f\boxtimes g)}{\partial{z_{ij}}}=\frac{1}{n} \left(\frac{\partial f}{\partial{x_i}}\boxtimes\frac{\partial g}{\partial{y_j}}\right).
\end{equation}
More generally, for $1\leq r\leq n$,
\begin{equation}\label{iterated-derivative-Kronecker}
\frac{\partial^{\,r}(f\boxtimes g)}{\partial z_{i_1j_1}\cdots\partial z_{i_rj_r}}=\frac{(n-r)!}{n!} \left(\frac{\partial^{\,r}f}{\partial x_{i_1}\cdots\partial x_{i_r}}\boxtimes\frac{\partial^{\,r}g}{\partial y_{j_1}\cdots\partial y_{j_r}}\right).
\end{equation}
\end{proposition}
\begin{proof}
It is enough to prove \eqref{first-derivative-Kronecker}, since \eqref{iterated-derivative-Kronecker} then follows by iteration. The symmetric $(n-1)$-linear form associated with the $\frac{\partial(f\boxtimes g)}{\partial{z_{ij}}}$ sends $U_1,\ldots,U_{n-1}\in V\otimes W$ to
\begin{equation*}
n\,T_{f\boxtimes g}(e_i\otimes\varepsilon_j,U_1,\ldots,U_{n-1}).
\end{equation*}
Evaluating this form on decomposable vectors $U_k=v_k\otimes w_k$ and using Definition~\ref{def:KroneckerProduct}, we obtain
\begin{equation*}
n\,T_{f\boxtimes g}(e_i\otimes\varepsilon_j,v_1\otimes w_1,\ldots,v_{n-1}\otimes w_{n-1})=n\,T_f(e_i,v_1,\ldots,v_{n-1})
T_g(\varepsilon_j,w_1,\ldots,w_{n-1}).
\end{equation*}
The symmetric $(n-1)$-linear forms associated with $\frac{\partial f}{\partial{x_i}}$ and $\frac{\partial g}{\partial{y_j}}$ send $(v_1,\ldots,v_{n-1})$ and $(w_1,\ldots,w_{n-1})$ to 
\begin{equation*}
n\,T_f(e_i,v_1,\ldots,v_{n-1}) \qquad\text{and}\qquad n\,T_g(\varepsilon_j,w_1,\ldots,w_{n-1}),
\end{equation*}
respectively. Hence the symmetric $(n-1)$-linear form associated with $\frac{\partial f}{\partial{x_i}}\boxtimes\frac{\partial g}{\partial{y_j}}$ takes the value
\begin{equation*}
n^2\,T_f(e_i,v_1,\ldots,v_{n-1})
\,T_g(\varepsilon_j,w_1,\ldots,w_{n-1})
\end{equation*}
on $(v_1\otimes w_1,\ldots,v_{n-1}\otimes w_{n-1})$. Therefore the two sides of \eqref{first-derivative-Kronecker} have the same associated symmetric $(n-1)$-linear form on every tuple of decomposable vectors. By Lemma~\ref{lem:Segre-Veronese-determination}, applied with
$n_k=m_k=1$ for $1\leq k\leq n-1$, the two symmetric
$(n-1)$-linear forms therefore agree identically. Hence the
corresponding homogeneous polynomials are equal, proving
\eqref{first-derivative-Kronecker}.
%Since decomposable tensors span $V\otimes W$, multilinearity implies that these forms agree on every tuple in $(V\otimes W)^{\times(n-1)}$. Hence the corresponding homogeneous polynomials are equal, proving \eqref{first-derivative-Kronecker}.
Repeated application of this identity gives the factor
\begin{equation*}
\frac{1}{n(n-1)\cdots(n-r+1)}=\frac{(n-r)!}{n!},
\end{equation*}
and proves \eqref{iterated-derivative-Kronecker}.
\end{proof}

Taking $r=2$ in \eqref{iterated-derivative-Kronecker} gives the global entrywise identity
\begin{equation}\label{Hessian-entrywise-Kronecker}
\frac{\partial^{2}(f\boxtimes g)}{\partial z_{ij}\partial z_{kl}}=\frac{1}{n(n-1)}\left(\frac{\partial^{2}f}{\partial x_i\partial x_k}\boxtimes\frac{\partial^{2}g}{\partial y_j\partial y_l}\right).
\end{equation}
The entrywise identity \eqref{Hessian-entrywise-Kronecker} is equivalent to the following global identity for Hessian matrices.

For matrices $A=(a_{ik})$ and $B=(b_{jl})$ whose entries are homogeneous polynomials on $V$ and $W$, respectively, we write $A\boxtimes B$ for the matrix indexed by pairs $(i,j)$ and $(k,l)$ whose $((i,j),(k,l))$-entry is the polynomial Kronecker product $a_{ik}\boxtimes b_{jl}$. After evaluation at a decomposable point $v\otimes w\in V\otimes W$, this specializes to the ordinary matrix Kronecker product $A(v)\boxtimes B(w)$.

\begin{corollary}[Hessian-matrix identity]
\label{cor:Hessian-matrix-identity}
Let $f\in\Sym^nV$ and $g\in\Sym^nW$. Then
\begin{equation}\label{Hessian-Kronecker}
\mathcal{H}_{f\boxtimes g}=\frac{1}{n(n-1)}\,\mathcal{H}_f\boxtimes\mathcal{H}_g.
\end{equation}
\end{corollary}
\begin{proof}
By the definition of Hessian matrix,
\begin{align*}
\big(\mathcal{H}_{f\boxtimes g}\big)_{(i,j),(k,l)} &=
\frac{\partial^{2}(f\boxtimes g)}{\partial z_{ij}\partial z_{kl}}, \\
\big(\mathcal{H}_f\boxtimes\mathcal{H}_g\big)_{(i,j),(k,l)} &=\frac{\partial^{2}f}{\partial x_i\partial x_k}\boxtimes\frac{\partial^{2}g}{\partial y_j\partial y_l}.
\end{align*}
The identity \eqref{Hessian-entrywise-Kronecker} shows that every entry of $\mathcal{H}_{f\boxtimes g}$ agrees with the corresponding entry of $\frac1{n(n-1)}\mathcal{H}_f\boxtimes\mathcal{H}_g$. Hence the two matrices are equal, proving \eqref{Hessian-Kronecker}.
\end{proof}

\begin{remark}
The identity \eqref{Hessian-Kronecker}, together with the symmetry of the Hessian matrices $\mathcal{H}_f$ and $\mathcal{H}_g$, gives
\begin{equation}
\big(\mathcal{H}_{f\boxtimes g}\big)_{(i,j),(k,l)}
=\big(\mathcal{H}_{f\boxtimes g}\big)_{(k,j),(i,l)}
=\big(\mathcal{H}_{f\boxtimes g}\big)_{(i,l),(k,j)}
=\big(\mathcal{H}_{f\boxtimes g}\big)_{(k,l),(i,j)}.
\end{equation}
Thus $\mathcal{H}_{f\boxtimes g}$ is a symmetric block matrix whose individual blocks are themselves symmetric.
\end{remark}

Although the Hessian-matrix identity \eqref{Hessian-Kronecker} holds globally on $V\otimes W$, taking determinants does not, in general, yield the corresponding identity for Hessian polynomials. After evaluation at a decomposable tensor $v\otimes w$, however, the right-hand side of \eqref{Hessian-Kronecker} becomes the ordinary Kronecker product of the evaluated Hessian matrices, so the Kronecker determinant formula applies. Thus the corresponding determinant identity holds on the Segre variety.

In general, equality on the Segre variety does not imply equality as polynomials on $V\otimes W$: their difference may be a nonzero element of the Segre ideal, which is generated by the $2\times2$ minors of the coordinate matrix $(z_{ij})$. Consequently, the corresponding multiplicative factorization of the Hessian polynomial need not hold on the whole tensor-product space. The persistent case has additional structure, which we exploit below.

\begin{lemma}[Hessian identity at decomposable points]\label{lem:Kronecker-Hessian-decomposable}
For every $v\in V$ and $w\in W$, one has
\begin{equation}\label{Hessian-Kronecker-decomposable}
\mathcal{H}_{f\boxtimes g}(v\otimes w)=
\frac{1}{n(n-1)}\,\mathcal{H}_f(v)\boxtimes\mathcal{H}_g(w)
\end{equation}
as bilinear forms on $V\otimes W$. Consequently,
\begin{equation}\label{det-Hessian-Kronecker-decomposable}
\Hess(f\boxtimes g)(v\otimes w)=
\left(\frac{1}{n(n-1)}\right)^{\!\!d_1d_2}
\Hess(f)(v)^{d_2}\Hess(g)(w)^{d_1}.
\end{equation}
\end{lemma}
\begin{proof}
Evaluating \eqref{Hessian-entrywise-Kronecker} at a decomposable point $v\otimes w$ and applying Definition~\ref{def:KroneckerProduct} to the homogeneous polynomials $\partial^{2}f/\partial x_i\partial x_k$ and $\partial^{2}g/\partial y_j\partial y_l$ gives
\begin{equation*}
\left(\frac{\partial^{2}f}{\partial x_i\partial x_k}
\boxtimes\frac{\partial^{2}g}{\partial y_j\partial y_l}
\right)(v\otimes w)=\frac{\partial^{2}f}{\partial x_i\partial x_k}(v)\,\frac{\partial^{2}g}{\partial y_j\partial y_l}(w).
\end{equation*}
Therefore, with rows and columns indexed by the pairs $(i,j)$ and $(k,l)$, respectively, identity \eqref{Hessian-Kronecker-decomposable} follows. Taking determinants in \eqref{Hessian-Kronecker-decomposable} and using $\det(A\boxtimes B)=\det(A)^{d_2}\det(B)^{d_1}$ gives
\begin{align*}
\Hess(f\boxtimes g)(v\otimes w)&=\det \left(\frac{1}{n(n-1)} \mathcal{H}_f(v)\boxtimes\mathcal{H}_g(w)
\right) \\
&=\left(\frac{1}{n(n-1)}\right)^{\!\!d_1d_2}
\Hess(f)(v)^{d_2}\Hess(g)(w)^{d_1},
\end{align*}
which proves \eqref{det-Hessian-Kronecker-decomposable}.
\end{proof}

\begin{corollary}\label{cor:Hessian-nonvanishing}
If $\Hess(f)\not\equiv 0$ and $\Hess(g)\not\equiv 0$, then $\Hess(f\boxtimes g)\not\equiv 0$.
\end{corollary}
\begin{proof}
Choose $v\in V$ and $w\in W$ such that $\Hess(f)(v)\neq 0$ and $\Hess(g)(w)\neq 0$. Equation~\eqref{det-Hessian-Kronecker-decomposable} then gives $\Hess(f\boxtimes g)(v\otimes w)\neq 0$.
\end{proof}

For a form in the distinguished isobaric weight space, the following result identifies the scalar appearing in its Hessian and shows that this scalar is multiplicative under Kronecker products.

\begin{proposition}[Multiplicativity of the Hessian coefficient]
\label{prop:hessian-coefficient}
Let $f\in\Sym^n\mathbbm{C}^{d_1}$ and $g\in\Sym^n\mathbbm{C}^{d_2}$ be isobaric of weights $d_1-1$ and $d_2-1$, respectively. Write
\begin{equation}\label{hessian-f-g}
\Hess(f) = \lambda_f\,x_0^{d_1(n-2)}, \qquad
\Hess(g) = \lambda_g\,y_0^{d_2(n-2)},
\end{equation}
where $\lambda_f,\lambda_g\in\mathbbm{C}$. Then
\begin{equation}\label{Hessian-product-coefficient}
\Hess(f\boxtimes g) = \lambda_{f\boxtimes g}\,z_{00}^{d_1d_2(n-2)},
\end{equation}
where
\begin{equation}\label{lambda-multiplicativity}
\lambda_{f\boxtimes g}=\left(\frac{1}{n(n-1)} \right)^{\!\!d_1d_2}\lambda_f^{d_2}\lambda_g^{d_1}.
\end{equation}
\end{proposition}
\begin{proof}
By Lemma~\ref{lem:Kronecker-isobaric}, the tensor
$f\boxtimes g$ is isobaric of weight $d_1d_2-1$. Hence
Lemma~\ref{lem:Hessian}, applied in dimension $d_1d_2$, gives $\Hess(f\boxtimes g)=\lambda_{f\boxtimes g}\,z_{00}^{d_1d_2(n-2)}$ for a uniquely determined scalar $\lambda_{f\boxtimes g}\in\mathbbm{C}$. Let $e_0\in\mathbbm{C}^{d_1}$ and
$\varepsilon_0\in\mathbbm{C}^{d_2}$ denote the first basis vectors. Since $z_{00}(e_0\otimes\varepsilon_0)=1$, equation~\eqref{Hessian-product-coefficient} gives $\lambda_{f\boxtimes g}=\Hess(f\boxtimes g)(e_0\otimes\varepsilon_0)$. Applying equation~\eqref{det-Hessian-Kronecker-decomposable} at $v=e_0$ and $w=\varepsilon_0$, we obtain
\begin{equation*}
\lambda_{f\boxtimes g}=\left(\frac{1}{n(n-1)}\right)^{\!\!d_1d_2}
\Hess(f)(e_0)^{d_2}\Hess(g)(\varepsilon_0)^{d_1}.
\end{equation*}
By equation~\eqref{hessian-f-g}, $\Hess(f)(e_0)=\lambda_f$ and $\Hess(g)(\varepsilon_0)=\lambda_g$. Substitution proves \eqref{lambda-multiplicativity}.
\end{proof}

Proposition~\ref{prop:hessian-coefficient} shows that the distinguished Hessian coefficient is multiplicative under the Kronecker product. The sufficient Hessian criterion of Theorem~\ref{thm:hessian} therefore yields the following theorem.

\begin{theorem}\label{thm:Persistent-Kronecker}
Let $f\in\Sym^n\mathbbm{C}^{d_1}$ and $g\in\Sym^n\mathbbm{C}^{d_2}$ be isobaric of weights $d_1-1$ and $d_2-1$, respectively. If $\Hess(f)\not\equiv 0$ and $\Hess(g)\not\equiv 0$, then $f\boxtimes g$ is persistent.
\end{theorem}
\begin{proof}
By Lemma~\ref{lem:Hessian} and the nonvanishing hypotheses, there exist nonzero scalars $\lambda_f,\lambda_g\in\mathbbm{C}^{\times}$ such that $\Hess(f)=\lambda_f\,x_0^{d_1(n-2)}$ and $\Hess(g)=\lambda_g\,y_0^{d_2(n-2)}$. By Proposition~\ref{prop:hessian-coefficient},
\begin{equation*}
\Hess(f\boxtimes g)=\left(\frac{1}{n(n-1)}\right)^{\!\!d_1d_2}
\lambda_f^{d_2}\lambda_g^{d_1}z_{00}^{d_1d_2(n-2)}.
\end{equation*}
Its coefficient is nonzero. Theorem~\ref{thm:hessian} therefore implies that $f\boxtimes g$ is persistent.
\end{proof}

\begin{corollary}[Closure under Kronecker products]\label{cor:persistent-isobaric-product}
Let $f\in\Sym^n\mathbbm{C}^{d_1}$ and $g\in\Sym^n\mathbbm{C}^{d_2}$ be persistent tensors that are isobaric of weights $d_1-1$ and
$d_2-1$, respectively. Then $f\boxtimes g$ is persistent.
\end{corollary}
\begin{proof}
By Lemma~\ref{lem:Hessian}, there exist scalars $\lambda_f,\lambda_g\in\mathbbm{C}$ such that $\Hess(f)=\lambda_f x_0^{d_1(n-2)}$ and $\Hess(g)=\lambda_g y_0^{d_2(n-2)}$. Since $f$ and $g$ are persistent, Theorem~\ref{thm:hessian}\,\textnormal{(iv)} implies that their Hessians are nonzero. Hence $\lambda_f,\lambda_g\neq0$, and the conclusion follows from Theorem~\ref{thm:Persistent-Kronecker}.
\end{proof}

\begin{example}\label{example-isobaric}
Consider the binary cubic form $f=x_0^2x_1\in\Sym^3\mathbbm{C}^2$ and the ternary cubic form $g=y_0^2y_2+y_0y_1^2\in\Sym^3\mathbbm{C}^3$. The form $f$ is isobaric of the distinguished weight $1=\dim\mathbbm{C}^2-1$, while both monomials of $g$ have weight $2=\dim\mathbbm{C}^3-1$, so $g$ is isobaric of the distinguished weight as well. Moreover, $\Hess(f)=-4x_0^2$ and $\Hess(g)=-8y_0^3$, and hence both $f$ and $g$ are persistent by Theorem~\ref{thm:hessian}. Let $z_{ij}$, with $0\leq i\leq1$ and $0\leq j\leq2$, denote the coordinates on $\mathbbm{C}^2\otimes\mathbbm{C}^3$ corresponding to $e_i\otimes\varepsilon_j$. With the restitution convention
\eqref{restitution-convention}, the nonzero coefficients of the associated symmetric trilinear forms are $T_f(e_0,e_0,e_1)=1/3$, $T_g(\varepsilon_0,\varepsilon_0,\varepsilon_2)=1/3$, and $T_g(\varepsilon_0,\varepsilon_1,\varepsilon_1)=1/3$, together with those obtained by symmetry. Consequently, $f\boxtimes g=\frac13 z_{00}^2z_{12}+\frac23 z_{00}z_{02}z_{10}+\frac23 z_{00}z_{01}z_{11}+\frac13 z_{01}^2z_{10}$. Although none of the four monomials occurring in this expression is persistent by itself, their resulting linear combination is persistent. Indeed, $f$ and $g$ are persistent and isobaric of their distinguished weights, so Corollary~\ref{cor:persistent-isobaric-product} applies. Finally, direct computation gives $\Hess(f\boxtimes g)=-\frac{64}{729}z_{00}^6$ in agreement with Proposition~\ref{prop:hessian-coefficient}. The Kronecker product $f\boxtimes g$ is illustrated in Figure~\ref{fig:Kronecker-Persistent}.
\end{example}

\begin{corollary}[Iterated Kronecker products] \label{cor:iterated-products}
For $1\leq j\leq k$, let $f_j\in\Sym^n\mathbbm{C}^{d_j}$ be persistent and isobaric of weight $d_j-1$. Then the iterated Kronecker product $\boxtimes_{j=1}^{k} f_j$ is persistent.
\end{corollary}
\begin{proof}
The assertion follows by induction. At each step, Lemma~\ref{lem:Kronecker-isobaric} shows that the partial Kronecker product remains isobaric of the distinguished weight in its ambient dimension, while Corollary~\ref{cor:persistent-isobaric-product} preserves persistence.
\end{proof}

\begin{corollary}[Kronecker powers]
\label{cor:Kronecker-powers}
Let $f\in\Sym^n\mathbbm{C}^d$ be persistent and isobaric of weight $d-1$. Then the Kronecker power $f^{\boxtimes k}$ is persistent for every integer $k\geq1$.
\end{corollary}
\begin{proof}
Apply Corollary~\ref{cor:iterated-products} with
$f_1=\cdots=f_k=f$.
\end{proof}

\begin{example}
Up to multiplication by a nonzero scalar, the binary $\mathsf{W}$-tensor, corresponding to the $n$-qubit state $|\mathsf{W}_n\rangle=\frac{1}{\sqrt{n}} \sum_{i=1}^{n}|0\rangle^{\otimes(i-1)} |1\rangle|0\rangle^{\otimes(n-i)}$ in quantum information theory \cite{DVC}, is represented by the polynomial $\mathsf{W}_n=x_0^{n-1}x_1$, which is isobaric of the distinguished weight. Corollary~\ref{cor:Kronecker-powers} therefore implies that every Kronecker power $\mathsf{W}_n^{\boxtimes k}$ is persistent. This yields an explicit infinite family of persistent tensors in
\begin{equation*}
\Sym^n\!\left((\mathbbm{C}^2)^{\otimes k}\right)
\simeq
\Sym^n\mathbbm{C}^{2^k}.
\end{equation*}
\end{example}

%%%%%%%%%%%%%%%%%%%%%%%%%%%%%%%%%%%%%%%%%%%%%%%%%%%%%%%%%%%%
%%%%%%%%%%%%%%%%%%%%%%%%%%%%%%%%%%%%%%%%%%%%%%%%%%%%%%%%%%%%
\section{Consequences in the classified cases}

The classification theorem of~\cite[Theorem~5]{GO25} provides normal forms for symmetric persistent tensors in the low-dimensional cases needed below. Up to linear equivalence and multiplication by a nonzero scalar, the normal form in $d=2$ is
\begin{equation}
f=x_0^{n-1}x_1,
\end{equation}
whereas in $d=3$ it is
\begin{equation}
f=x_0^{n-2}(x_0x_2+x_1^2).
\end{equation}
For cubic forms in $d=4$, the classified family is represented by
\begin{equation}
f=\lambda_1x_0^2x_3+\lambda_2x_0x_1x_2+\lambda_3x_1^3,
\qquad
\lambda_1\lambda_2\neq 0,
\end{equation}
which is isobaric of weight three.

For an invertible linear map $A$ on the ambient vector space, we use the convention
\begin{equation}
(A\cdot f)(x)=f(A^{-1}x).
\end{equation}
Conciseness and persistence are invariant under this action. Indeed, an invertible change of coordinates induces an invertible change on the space of first derivatives and carries the exceptional hyperplane in the definition of persistence to another hyperplane. The assertion follows inductively on the degree. Persistence is also unchanged by multiplication by a nonzero scalar.

Moreover, if $A\in\operatorname{GL}(V)$ and
$B\in\operatorname{GL}(W)$, then the Kronecker product is equivariant:
\begin{equation}\label{Kronecker-equivariance}
(A\cdot f)\boxtimes(B\cdot g)=(A\boxtimes B)\cdot(f\boxtimes g).
\end{equation}

\begin{corollary}[Closure in the classified cases]\label{cor:closure-classified}
The following statements hold:
\begin{enumerate}[label=\textup{(\roman*)}]
\item If $d_1,d_2\leq3$, then the Kronecker product of two persistent tensors in $\Sym^n\mathbbm{C}^{d_1}$ and
$\Sym^n\mathbbm{C}^{d_2}$ is persistent.
\item If $n=3$ and $d_1,d_2\leq4$, then the Kronecker product of two persistent tensors in $\Sym^3\mathbbm{C}^{d_1}$ and $\Sym^3\mathbbm{C}^{d_2}$ is persistent.
\end{enumerate}
\end{corollary}
\begin{proof}
In the stated ranges, \cite[Theorem~5]{GO25} shows that each persistent tensor is, up to multiplication by a nonzero scalar and an invertible linear change of coordinates, equivalent to a persistent tensor that is isobaric of the distinguished weight. Let $\widetilde f$ and $\widetilde g$ be the corresponding normal-form representatives. By Corollary~\ref{cor:persistent-isobaric-product}, $\widetilde f\boxtimes\widetilde g$ is persistent. Equation~\eqref{Kronecker-equivariance} shows that the Kronecker product of the original tensors is obtained from $\widetilde f\boxtimes\widetilde g$ by an invertible linear change of coordinates and multiplication by a nonzero scalar. Since both operations preserve persistence, the conclusion follows.
\end{proof}
%%%%%%%%%%%%%%%%%%%%%%%%%%%%%%%%%%%%%%%%%%%%%%%%%%%%%%%%%%%%
%%%%%%%%%%%%%%%%%%%%%%%%%%%%%%%%%%%%%%%%%%%%%%%%%%%%%%%%%%%%
\section{Normalized Hessian spaces and triangularizability}
\label{sec:matrix-space}
%%%%%%%%%%%%%%%%%%%%%%%%%%%%%%%%%%%%%%%
\subsection{The cubic case}
The cubic case admits a complementary matrix-theoretic interpretation. Linear spaces consisting entirely of nilpotent matrices have been studied extensively; see, for instance, \cite{MOR91,MORS21}. A fundamental distinction in this theory is that a nilpotent matrix space need not be simultaneously strictly triangularizable. In the present setting, this distinction arises naturally from the normalized Hessian spaces associated with persistent cubic forms.

In this section, we describe a sufficient triangularizability mechanism for the perfect-power Hessian phenomenon. The counterexample of Section~\ref{sec:counterexample} shows that such triangularizability does not hold for every persistent cubic, but the criterion below gives a useful sufficient condition for Kronecker closure.

Let $f\in\Sym^3V$ be persistent, where $d=\dim V$. By Theorem~\ref{thm:hessian}, there exist a nonzero linear form $\ell_f\in V^\vee$ and a scalar $\lambda_f\in\mathbbm{C}^\times$ such that
\begin{equation}
\Hess(f)=\lambda_f\,\ell_f^d.
\end{equation}
Choose $v_0\in V$ satisfying $\ell_f(v_0)=1$. Since $\det\mathcal{H}_f(v_0)=\lambda_f\neq0$, we may define the normalized Hessian map
\begin{equation}\label{normalized-hessian-map}
\mathcal{N}_f(v)=\mathcal{H}_f(v_0)^{-1}\mathcal{H}_f(v)-\ell_f(v)I_V.
\end{equation}
The restriction
\begin{equation}
\mathcal{N}_f(\ker\ell_f)\subseteq\operatorname{End}(V)
\end{equation}
is a linear space consisting entirely of nilpotent endomorphisms. Indeed, if $u\in\ker\ell_f$, then for every $t\in\mathbbm{C}$,
\begin{equation}
\det\big(I_V+t\mathcal{N}_f(u)\big)=
\frac{\det\mathcal{H}_f(v_0+tu)}{\det\mathcal{H}_f(v_0)}=
\frac{\lambda_f\,\ell_f(v_0+tu)^d}{\lambda_f}=1.
\end{equation}
Thus all elementary symmetric functions of the eigenvalues of $\mathcal{N}_f(u)$ vanish. Consequently, $\det(\mu I_V-\mathcal{N}_f(u))=\mu^d$, so the characteristic polynomial of $\mathcal{N}_f(u)$ is $\mu^d$. By the Cayley--Hamilton theorem $\mathcal{N}_f(u)^d=0$, and hence $\mathcal{N}_f(u)$ is nilpotent.

The following result shows how simultaneous strict triangularizability of these normalized Hessian spaces gives a direct matrix-theoretic explanation for the perfect-power Hessian of a Kronecker product.

\begin{proposition}[Triangularizability criterion] \label{prop:cubic-triangularizability}
Let $f\in\Sym^3V$ and $g\in\Sym^3W$ be persistent cubic forms, where $\dim V=d_1$ and $\dim W=d_2$. Let $\ell_f\in V^\vee$ and $\ell_g\in W^\vee$ be nonzero linear forms such that for some $\lambda_f,\lambda_g\in\mathbbm{C}^\times$,
\begin{equation}\label{cubic-factor-hessians}
\Hess(f)=\lambda_f\,\ell_f^{d_1},
\qquad
\Hess(g)=\lambda_g\,\ell_g^{d_2}.
\end{equation}
Choose $v_0\in V$ and $w_0\in W$ satisfying
$\ell_f(v_0)=\ell_g(w_0)=1$, and define
\begin{equation}\label{normalized-hessian-spaces}
\mathcal{N}_f(v)=\mathcal{H}_f(v_0)^{-1}\mathcal{H}_f(v)-\ell_f(v)I_V,
\qquad
\mathcal{N}_g(w)=\mathcal{H}_g(w_0)^{-1}\mathcal{H}_g(w)-\ell_g(w)I_W.
\end{equation}
If the matrix spaces $\mathcal{N}_f(\ker\ell_f)\subseteq\operatorname{End}(V)$ and $\mathcal{N}_g(\ker\ell_g)\subseteq\operatorname{End}(W)$ are simultaneously strictly triangularizable, then $f\boxtimes g$ is persistent.
\end{proposition}
\begin{proof}
Since $\ell_f(v_0)=\ell_g(w_0)=1$, equation~\eqref{cubic-factor-hessians} gives $\det\mathcal{H}_f(v_0)=\lambda_f\neq0$ and $\det\mathcal{H}_g(w_0)=\lambda_g\neq0$; hence the inverses in \eqref{normalized-hessian-spaces} are well defined. Since $f$ and $g$ are cubic, Corollary~\ref{cor:Hessian-matrix-identity} gives the global identity
\begin{equation}\label{cubic-global-Hessian}
\mathcal{H}_{f\boxtimes g}=\frac{1}{6}\,
\mathcal{H}_f\boxtimes\mathcal{H}_g.
\end{equation}
Moreover, by \eqref{normalized-hessian-spaces},
\begin{equation}\label{normalized-Hessian-factorization}
\mathcal{H}_f(v)=\mathcal{H}_f(v_0)\big(\ell_f(v)I_V+\mathcal{N}_f(v)\big),
\qquad
\mathcal{H}_g(w)=\mathcal{H}_g(w_0)\big(\ell_g(w)I_W+\mathcal{N}_g(w)\big).
\end{equation}
Let $U\in V\otimes W$ be arbitrary and choose a decomposition $U=\sum_{\alpha=1}^{r}v_\alpha\otimes w_\alpha$. Since the Hessian matrices of cubic forms depend linearly on their
arguments, evaluating \eqref{cubic-global-Hessian} at $U$ gives
\begin{equation}\label{cubic-Hessian-U}
\mathcal{H}_{f\boxtimes g}(U)=\frac{1}{6}\sum_{\alpha=1}^{r}
\mathcal{H}_f(v_\alpha)\boxtimes\mathcal{H}_g(w_\alpha).
\end{equation}
Substituting \eqref{normalized-Hessian-factorization} into
\eqref{cubic-Hessian-U}, we obtain
\begin{equation}\label{cubic-Hessian-factorization}
\mathcal{H}_{f\boxtimes g}(U)=\frac{1}{6}\big(\mathcal{H}_f(v_0)\boxtimes\mathcal{H}_g(w_0)\big)\mathcal{M}(U),
\end{equation}
where
\begin{equation}\label{matrix-MU}
\mathcal{M}(U)=\sum_{\alpha=1}^{r} \big(\ell_f(v_\alpha)I_V+\mathcal{N}_f(v_\alpha)\big)
\boxtimes\big(\ell_g(w_\alpha)I_W+\mathcal{N}_g(w_\alpha)\big).
\end{equation}
Set $\ell=\ell_f\boxtimes\ell_g\in(V\otimes W)^\vee$. Since $\ell(U)=\sum_{\alpha=1}^{r}\ell_f(v_\alpha)\ell_g(w_\alpha)$, the scalar part of $\mathcal{M}(U)$ is $\ell(U)I_{V\otimes W}$. Now let $v\in V$. Since $\ell_f(v_0)=1$, we may write $v=\ell_f(v)v_0+u$, where $u\in\ker\ell_f$. The linearity of $\mathcal{N}_f$, together with $\mathcal{N}_f(v_0)=0$, gives $\mathcal{N}_f(v)=\mathcal{N}_f(u)$. Thus, in a basis simultaneously strictly triangularizing $\mathcal{N}_f(\ker\ell_f)$, every $\mathcal{N}_f(v)$ is strictly upper triangular. The analogous statement holds for every $\mathcal{N}_g(w)$. Expanding each summand in \eqref{matrix-MU} gives
\begin{align*}
\big(\ell_f(v_\alpha)I_V+\mathcal{N}_f(v_\alpha)\big)
\!\boxtimes\!\big(\ell_g(w_\alpha)I_W+\mathcal{N}_g(w_\alpha)\big)\!&=
\ell_f(v_\alpha)\ell_g(w_\alpha)I_{V\otimes W}
+\ell_g(w_\alpha)\mathcal{N}_f(v_\alpha)\boxtimes I_W \\
&\,\,\,\,+\ell_f(v_\alpha)\,I_V\boxtimes \mathcal{N}_g(w_\alpha)
+\mathcal{N}_f(v_\alpha)\boxtimes \mathcal{N}_g(w_\alpha).
\end{align*}
With respect to the induced lexicographically ordered basis of
$V\otimes W$, the last three terms are strictly upper triangular.
Consequently, $\mathcal{M}(U)=\ell(U)I_{V\otimes W}+R(U)$ where $R(U)$ is strictly upper triangular. Hence
\begin{equation}\label{det-MU}
\det\mathcal{M}(U)=\ell(U)^{d_1d_2}.
\end{equation}
Taking determinants in \eqref{cubic-Hessian-factorization}, using $\det\mathcal{H}_f(v_0)=\lambda_f$ and $\det\mathcal{H}_g(w_0)=\lambda_g$, and applying \eqref{det-MU}, we obtain
\begin{align*}
\Hess(f\boxtimes g)(U)
&=\Big(\frac{1}{6}\Big)^{\!d_1d_2}
\det\big(\mathcal{H}_f(v_0)\boxtimes\mathcal{H}_g(w_0)\big)
\det\mathcal{M}(U)\\
&=\Big(\frac{1}{6}\Big)^{\!d_1d_2}
\lambda_f^{d_2}\lambda_g^{d_1}\ell(U)^{d_1d_2}.
\end{align*}
The coefficient is nonzero, so the Hessian of $f\boxtimes g$ is a nonzero $d_1d_2$-th power of a linear form. Hence $f\boxtimes g$ is persistent.
\end{proof}

The preceding argument can in fact be made one-sided: simultaneous strict triangularizability of the normalized Hessian space of only one factor is sufficient.

\begin{corollary}[One-sided triangularizability criterion]
\label{cor:one-sided-triangularizability}
Under the hypotheses and notation of Proposition~\ref{prop:cubic-triangularizability}, it is sufficient to assume that at least one of the matrix spaces $\mathcal{N}_f(\ker\ell_f)\subseteq\operatorname{End}(V)$ and $\mathcal{N}_g(\ker\ell_g)\subseteq\operatorname{End}(W)$ is simultaneously strictly triangularizable. Then $f\boxtimes g$ is persistent.
\end{corollary}
\begin{proof}
By symmetry, assume that $\mathcal{N}_g(\ker\ell_g)$ is simultaneously strictly triangularizable. Let $U\in V\otimes W$ be arbitrary and choose a decomposition $U=\sum_{\alpha=1}^{r}v_\alpha\otimes w_\alpha$. As in the proof of Proposition~\ref{prop:cubic-triangularizability},
\begin{equation*}
\mathcal{H}_{f\boxtimes g}(U)=\frac{1}{6}\big(\mathcal{H}_f(v_0) \boxtimes \mathcal{H}_g(w_0)\big)\mathcal{M}(U),
\end{equation*}
where
\begin{equation*}
\mathcal{M}(U)=\sum_{\alpha=1}^{r}
\big(\ell_f(v_\alpha)I_V+\mathcal{N}_f(v_\alpha)\big)\boxtimes\big(\ell_g(w_\alpha)I_W+\mathcal{N}_g(w_\alpha)\big).
\end{equation*}
Since $\mathcal{N}_g(w_0)=0$, every $w\in W$ can be written as $w=\ell_g(w)w_0+u$ with $u\in\ker\ell_g$, and hence $\mathcal{N}_g(w)=\mathcal{N}_g(u)$. Thus there exists a basis of $W$ in which every $\mathcal{N}_g(w)$ is strictly upper triangular. In this basis, each endomorphism $\ell_g(w)I_W+\mathcal{N}_g(w)$ is upper triangular with all diagonal entries equal to $\ell_g(w)$. Consequently, with respect to the induced product basis of $V\otimes W$, the endomorphism $\mathcal{M}(U)$ is block upper triangular with $d_2$ identical diagonal blocks
\begin{equation*}
\mathcal{B}(U)=\sum_{\alpha=1}^{r}\ell_g(w_\alpha)
\big(\ell_f(v_\alpha)I_V+\mathcal{N}_f(v_\alpha)\big).
\end{equation*}
Set $v_U=\sum_{\alpha=1}^{r}\ell_g(w_\alpha)v_\alpha$.  By the linearity of $\ell_f$ and $\mathcal{N}_f$,
\begin{equation*}
\mathcal{B}(U)=\ell_f(v_U)I_V+\mathcal{N}_f(v_U)=\mathcal{H}_f(v_0)^{-1}\mathcal{H}_f(v_U).
\end{equation*}
Therefore,
\begin{equation*}
\det\mathcal{B}(U)=\frac{\det\mathcal{H}_f(v_U)}{\det\mathcal{H}_f(v_0)}=\frac{\lambda_f\,\ell_f(v_U)^{d_1}}{\lambda_f}=\ell_f(v_U)^{d_1}.
\end{equation*}
Since $\mathcal{M}(U)$ has $d_2$ identical diagonal blocks,
\begin{equation*}
\det\mathcal{M}(U)=\det\mathcal{B}(U)^{d_2}=\ell_f(v_U)^{d_1d_2}.
\end{equation*}
Setting $\ell=\ell_f\boxtimes\ell_g\in(V\otimes W)^\vee$, one has
\begin{equation*}
\ell_f(v_U)=\sum_{\alpha=1}^{r}\ell_f(v_\alpha)\ell_g(w_\alpha) =\ell(U).
\end{equation*}
Taking determinants in the factorization of
$\mathcal{H}_{f\boxtimes g}(U)$ gives
\begin{equation*}
\Hess(f\boxtimes g)(U)=\Big(\frac{1}{6}\Big)^{\!d_1d_2} \lambda_f^{d_2}\lambda_g^{d_1}\ell(U)^{d_1d_2},
\end{equation*}
and hence $f\boxtimes g$ is persistent.
\end{proof}

\begin{remark}\label{rem:cubic-Hessian-insufficient}
For persistent cubic forms, the perfect-power identities $\Hess(f)=\lambda_f\ell_f^{d_1}$ and $\Hess(g)=\lambda_g\ell_g^{d_2}$ determine the expected Hessian factorization of $f\boxtimes g$ on the Segre variety, but do not control the determinant at nondecomposable tensors, where the Hessian matrix is a sum of Kronecker products. The triangularizability hypotheses in Proposition~\ref{prop:cubic-triangularizability} and Corollary~\ref{cor:one-sided-triangularizability} provide the additional matrix-theoretic structure needed to obtain a global perfect-power identity. As Section~\ref{sec:counterexample} shows, such additional structure cannot be omitted in general.
\end{remark}

%%%%%%%%%%%%%%%%%%%%%%%%%%%%%%%%%%%%%%%
\subsection{Extension to arbitrary degree}
The preceding construction admits a natural extension to symmetric persistent forms of arbitrary degree. For cubic forms, the Hessian matrix depends linearly on a single vector. In higher degree, the corresponding object is the Hessian matrix of the $(n-2)$-fold partial polarization.

Let $f\in\Sym^nV$, where $n\geq3$, and put $r=n-2$. For $\mathbf{v}=(v^{(1)},\ldots,v^{(r)})\in V^{\times r}$, write $f_{\mathbf{v}}\coloneqq f_{v^{(1)},\ldots,v^{(r)}}$ and define the \emph{polarized Hessian matrix}
\begin{equation}\label{polarized-Hessian-map}
\mathcal{H}^{\mathrm{pol}}_f(\mathbf{v})
\coloneqq\mathcal{H}_{f_{\mathbf{v}}}.
\end{equation}
Thus $\mathcal{H}^{\mathrm{pol}}_f$ is a symmetric multilinear map from $V^{\times r}$ to the space of symmetric bilinear forms on $V$. By \eqref{partial-polarization} and the restitution convention \eqref{restitution-convention}, one has
\begin{equation}\label{polarized-Hessian-Tf}
\mathcal{H}^{\mathrm{pol}}_f(v^{(1)},\ldots,v^{(r)})(a,c)
=n!\,T_f(v^{(1)},\ldots,v^{(r)},a,c)
\end{equation}
for every $a,c\in V$.

Suppose that $f$ is persistent and let $d=\dim V$. By Theorem~\ref{thm:hessian}, there exists a nonzero multihomogeneous polynomial $P_f$ of multidegree $(1,\ldots,1)$ such that
\begin{equation}\label{polarized-Hessian-perfect-power}
\Hess(f_{\mathbf{v}})\coloneqq\det\mathcal{H}^{\mathrm{pol}}_f(\mathbf{v})=P_f(\mathbf{v})^d.
\end{equation}
Since $P_f$ has degree one in each of its $r$ vector arguments, it may equivalently be regarded as an $r$-linear form on $V^{\times r}$.

Choose $\mathbf{v}_0\in V^{\times r}$ such that $P_f(\mathbf{v}_0)=1$. Then $\mathcal{H}^{\mathrm{pol}}_f(\mathbf{v}_0)$ is invertible, and we define the \emph{normalized polarized Hessian map}
\begin{equation}\label{normalized-polarized-Hessian}
\mathcal{N}^{\mathrm{pol}}_f(\mathbf{v})\coloneqq
\mathcal{H}^{\mathrm{pol}}_f(\mathbf{v}_0)^{-1}
\mathcal{H}^{\mathrm{pol}}_f(\mathbf{v})
-P_f(\mathbf{v})I_V.
\end{equation}
The map $\mathcal{N}^{\mathrm{pol}}_f:V^{\times r}\to\operatorname{End}(V)$ is multilinear. We denote by
\begin{equation}\label{normalized-polarized-Hessian-space}
\mathscr{N}^{\mathrm{pol}}_f\coloneqq\operatorname{Span}\big\{\mathcal{N}^{\mathrm{pol}}_f(\mathbf{v})\mid\mathbf{v}\in V^{\times r}\big\}\subseteq\operatorname{End}(V)
\end{equation}
the associated normalized polarized Hessian space.

The following identity is the higher-degree analogue of \eqref{cubic-global-Hessian}.

\begin{lemma}[Polarized Hessian identity]
\label{lem:polarized-Hessian-Kronecker}
Let $f\in\Sym^nV$ and $g\in\Sym^nW$, where $n\geq3$, and put $r=n-2$. Then, for arbitrary $v^{(j)}\in V$ and $w^{(j)}\in W$, $1\leq j\leq r$,
\begin{equation}\label{polarized-Hessian-Kronecker}
\mathcal{H}^{\mathrm{pol}}_{f\boxtimes g}
\big(v^{(1)}\otimes w^{(1)},\ldots,
v^{(r)}\otimes w^{(r)}\big)=\frac{1}{n!}\,
\mathcal{H}^{\mathrm{pol}}_f(v^{(1)},\ldots,v^{(r)})
\boxtimes
\mathcal{H}^{\mathrm{pol}}_g(w^{(1)},\ldots,w^{(r)}).
\end{equation}
\end{lemma}
\begin{proof}
Let $a,c\in V$ and $b,d\in W$. By \eqref{polarized-Hessian-Tf}, both sides of \eqref{polarized-Hessian-Kronecker}, when evaluated on $a\otimes b$ and $c\otimes d$, contain the common nonzero factor $n!$. Dividing by this factor, it is enough to verify that
\begin{equation*}
T_{f\boxtimes g}\big(v^{(1)}\otimes w^{(1)},\ldots,
v^{(r)}\otimes w^{(r)},a\otimes b,c\otimes d\big)
=T_f(v^{(1)},\ldots,v^{(r)},a,c)\,
T_g(w^{(1)},\ldots,w^{(r)},b,d),
\end{equation*}
which follows directly from Definition~\ref{def:KroneckerProduct}. Thus \eqref{polarized-Hessian-Kronecker} holds on pairs of decomposable vectors. By Lemma~\ref{lem:Segre-Veronese-determination}, the two bilinear forms therefore agree identically on $V\otimes W$, and hence \eqref{polarized-Hessian-Kronecker} follows.
\end{proof}

For multilinear forms $P_f:V^{\times r}\to\mathbbm{C}$ and $P_g:W^{\times r}\to\mathbbm{C}$, let
\begin{equation}
P_f\boxtimes P_g:(V\otimes W)^{\times r} \longrightarrow\mathbbm{C}
\end{equation}
denote the multilinear form determined by
\begin{equation}\label{PfPg-Kronecker}
(P_f\boxtimes P_g)
\big(v^{(1)}\otimes w^{(1)},\ldots,
v^{(r)}\otimes w^{(r)}\big)=
P_f(v^{(1)},\ldots,v^{(r)})P_g(w^{(1)},\ldots,w^{(r)}).
\end{equation}

\begin{theorem}[Triangularizability criterion in arbitrary degree]
\label{thm:general-triangularizability}
Let $f\in\Sym^nV$ and $g\in\Sym^nW$ be persistent, where $n\geq3$, $\dim V=d_1$, and $\dim W=d_2$. Let $P_f$ and $P_g$ be the multilinear forms determined by
\begin{equation}
\Hess(f_{\mathbf{v}})=P_f(\mathbf{v})^{d_1},
\qquad
\Hess(g_{\mathbf{w}})=P_g(\mathbf{w})^{d_2},
\end{equation}
and let $\mathscr{N}^{\mathrm{pol}}_f\subseteq\operatorname{End}(V)$ and $\mathscr{N}^{\mathrm{pol}}_g\subseteq\operatorname{End}(W)$ be the normalized polarized Hessian spaces defined by \eqref{normalized-polarized-Hessian}--\eqref{normalized-polarized-Hessian-space}. If both $\mathscr{N}^{\mathrm{pol}}_f$ and $\mathscr{N}^{\mathrm{pol}}_g$ are simultaneously strictly triangularizable, then $f\boxtimes g$ is persistent.
\end{theorem}
\begin{proof}
Put $r=n-2$. Choose $\mathbf{v}_0\in V^{\times r}$ and $\mathbf{w}_0\in W^{\times r}$ such that $P_f(\mathbf{v}_0)=P_g(\mathbf{w}_0)=1$. By definition,
\begin{align}\label{general-normalized-factorization-1}
\mathcal{H}^{\mathrm{pol}}_f(\mathbf{v})&=
\mathcal{H}^{\mathrm{pol}}_f(\mathbf{v}_0)\big(P_f(\mathbf{v})I_V+\mathcal{N}^{\mathrm{pol}}_f(\mathbf{v})\big), \\ \label{general-normalized-factorization-2}
\mathcal{H}^{\mathrm{pol}}_g(\mathbf{w})&=\mathcal{H}^{\mathrm{pol}}_g(\mathbf{w}_0)\big(P_g(\mathbf{w})I_W+\mathcal{N}^{\mathrm{pol}}_g(\mathbf{w})\big).
\end{align}
Let $U^{(1)},\ldots,U^{(r)}\in V\otimes W$ be arbitrary. For each $j=1,\ldots,r$, choose a decomposition
\begin{equation*}
U^{(j)}=\sum_{\alpha_j}v_{\alpha_j}^{(j)}\otimes w_{\alpha_j}^{(j)}.
\end{equation*}
Since the polarized Hessian is multilinear in its $r$ arguments, Lemma~\ref{lem:polarized-Hessian-Kronecker} gives
\begin{equation}\label{general-polarized-Hessian-U}
\mathcal{H}^{\mathrm{pol}}_{f\boxtimes g}
(U^{(1)},\ldots,U^{(r)})=\frac{1}{n!}
\sum_{\alpha_1,\ldots,\alpha_r}
\mathcal{H}^{\mathrm{pol}}_f(v_{\alpha_1}^{(1)},\ldots,v_{\alpha_r}^{(r)})\boxtimes\mathcal{H}^{\mathrm{pol}}_g
(w_{\alpha_1}^{(1)},\ldots,w_{\alpha_r}^{(r)}).
\end{equation}
For brevity, put $\mathbf{v}_{\boldsymbol{\alpha}}=(v_{\alpha_1}^{(1)},\ldots,v_{\alpha_r}^{(r)})$ and $\mathbf{w}_{\boldsymbol{\alpha}}=(w_{\alpha_1}^{(1)},\ldots,w_{\alpha_r}^{(r)})$. Substituting \eqref{general-normalized-factorization-1} and \eqref{general-normalized-factorization-2} into \eqref{general-polarized-Hessian-U}, we obtain
\begin{equation}\label{general-polarized-Hessian-factorization}
\mathcal{H}^{\mathrm{pol}}_{f\boxtimes g}
(U^{(1)},\ldots,U^{(r)})=\frac{1}{n!}
\big(\mathcal{H}^{\mathrm{pol}}_f(\mathbf{v}_0)\boxtimes \mathcal{H}^{\mathrm{pol}}_g(\mathbf{w}_0)\big)\,\mathcal{M}(U^{(1)},\ldots,U^{(r)}),
\end{equation}
where
\begin{equation}\label{general-M-matrix}
\mathcal{M}(U^{(1)},\ldots,U^{(r)})=
\sum_{\alpha_1,\ldots,\alpha_r}
\big(P_f(\mathbf{v}_{\boldsymbol{\alpha}})I_V+
\mathcal{N}^{\mathrm{pol}}_f(\mathbf{v}_{\boldsymbol{\alpha}})\big) \boxtimes
\big(P_g(\mathbf{w}_{\boldsymbol{\alpha}})I_W+
\mathcal{N}^{\mathrm{pol}}_g(\mathbf{w}_{\boldsymbol{\alpha}})\big).
\end{equation}
By hypothesis, choose bases of $V$ and $W$ simultaneously strictly triangularizing $\mathscr{N}^{\mathrm{pol}}_f$ and $\mathscr{N}^{\mathrm{pol}}_g$, respectively. Hence every $\mathcal{N}^{\mathrm{pol}}_f(\mathbf{v})$ and every $\mathcal{N}^{\mathrm{pol}}_g(\mathbf{w})$ is strictly upper triangular in these bases. Expanding a summand of \eqref{general-M-matrix}, the nonscalar terms are
\begin{equation*}
P_g(\mathbf{w})\,\mathcal{N}^{\mathrm{pol}}_f(\mathbf{v})
\boxtimes I_W,
\qquad
P_f(\mathbf{v})\,I_V\boxtimes
\mathcal{N}^{\mathrm{pol}}_g(\mathbf{w}),
\qquad
\mathcal{N}^{\mathrm{pol}}_f(\mathbf{v})\boxtimes
\mathcal{N}^{\mathrm{pol}}_g(\mathbf{w}).
\end{equation*}
With respect to the induced lexicographically ordered basis of $V\otimes W$, each of these matrices is strictly upper triangular. Consequently,
\begin{equation}\label{general-M-upper}
\mathcal{M}(U^{(1)},\ldots,U^{(r)})=
L(U^{(1)},\ldots,U^{(r)})I_{V\otimes W}+R,
\end{equation}
where $R$ is strictly upper triangular and
\begin{equation*}
L(U^{(1)},\ldots,U^{(r)})=\sum_{\alpha_1,\ldots,\alpha_r}P_f(\mathbf{v}_{\boldsymbol{\alpha}})P_g(\mathbf{w}_{\boldsymbol{\alpha}})=(P_f\boxtimes P_g)(U^{(1)},\ldots,U^{(r)}).
\end{equation*}
Therefore
\begin{equation}\label{general-det-M}
\det\mathcal{M}(U^{(1)},\ldots,U^{(r)})=
(P_f\boxtimes P_g)(U^{(1)},\ldots,U^{(r)})^{d_1d_2}.
\end{equation}
Moreover,
\begin{equation*}
\Hess(f_{\mathbf{v_0}})=P_f(\mathbf{v}_0)^{d_1}=1,
\qquad
\Hess(g_{\mathbf{w_0}})=P_g(\mathbf{w}_0)^{d_2}=1.
\end{equation*}
Taking determinants in \eqref{general-polarized-Hessian-factorization} and using \eqref{general-det-M}, we obtain
\begin{equation}\label{general-perfect-power}
\det\mathcal{H}^{\mathrm{pol}}_{f\boxtimes g}
(U^{(1)},\ldots,U^{(r)})=
\left(\frac{1}{n!}\right)^{\!d_1d_2}
(P_f\boxtimes P_g)(U^{(1)},\ldots,U^{(r)})^{d_1d_2}.
\end{equation}
Since $P_f$ and $P_g$ are nonzero, so is $P_f\boxtimes P_g$. Hence the right-hand side of \eqref{general-perfect-power} is the $d_1d_2$-th power of a nonzero multihomogeneous polynomial of multidegree $(1,\ldots,1)$. By Theorem~\ref{thm:hessian}, it follows that $f\boxtimes g$ is persistent.
\end{proof}

\begin{remark}
For $n=3$, one has $r=1$, and the multihomogeneous polynomials $P_f$ and $P_g$ are linear forms. After the normalization $P_f=\ell_f$ and $P_g=\ell_g$, the construction above reduces to the normalized Hessian construction of Proposition~\ref{prop:cubic-triangularizability}. Thus Theorem~\ref{thm:general-triangularizability} is the arbitrary-degree analogue of the two-sided cubic triangularizability criterion.
\end{remark}

\begin{proposition}[Triangularizability of the distinguished isobaric class]
\label{prop:isobaric-triangularizable}
Let $f\in\Sym^n\mathbbm{C}^d$, with $n\geq3$, be isobaric of weight $d-1$ and suppose that $\Hess(f)\not\equiv0$. Then $f$ is persistent. Moreover, the reference tuple $\mathbf{v}_0$ in \eqref{normalized-polarized-Hessian} can be chosen so that the corresponding normalized polarized Hessian space $\mathscr{N}^{\mathrm{pol}}_f$ is simultaneously strictly triangularizable.
\end{proposition}
\begin{proof}
Put $r=n-2$, let $e_0,\ldots,e_{d-1}$ be the standard basis of $\mathbbm{C}^d$, and, for $\boldsymbol{k}=(k_1,\ldots,k_r)\in\{0,\ldots,d-1\}^r$, write $\mathbf{e}_{\boldsymbol{k}}\coloneqq(e_{k_1},\ldots,e_{k_r})$, $|\boldsymbol{k}|\coloneqq k_1+\cdots+k_r$, and $\boldsymbol{0}=(0,\ldots,0)$. Since $f$ is isobaric of weight $d-1$, its associated symmetric
$n$-linear form $T_f$ satisfies
\begin{equation*}
T_f(e_{i_1},\ldots,e_{i_n})\neq0
\quad\Longrightarrow\quad
i_1+\cdots+i_n=d-1.
\end{equation*}
By \eqref{polarized-Hessian-Tf}, it follows that
\begin{equation}\label{isobaric-pol-Hessian-support}
\big(\mathcal{H}^{\mathrm{pol}}_f
(\mathbf{e}_{\boldsymbol{k}})\big)_{ij}\neq0
\quad \Longrightarrow \quad i+j+|\boldsymbol{k}|=d-1.
\end{equation}
In particular, $\mathcal{H}^{\mathrm{pol}}_f(\mathbf{e}_{\boldsymbol{0}})$ is supported on the anti-diagonal $i+j=d-1$. By Lemma~\ref{lem:Hessian}, there exists $\lambda\in\mathbbm{C}^{\times}$ such that $\Hess(f)=\lambda x_0^{d(n-2)}$. Moreover, by \eqref{polarized-Hessian-Tf} and \eqref{hessian-polarization}, $\mathcal{H}^{\mathrm{pol}}_f(x,\ldots,x)=(n-2)!\,\mathcal{H}_f(x)$. Hence
\begin{equation*}
\Hess(f_{\mathbf{e}_{\boldsymbol{0}}})=((n-2)!)^d\Hess(f)(e_0)=((n-2)!)^d\lambda\neq0.
\end{equation*}
Thus $\mathcal{H}^{\mathrm{pol}}_f(\mathbf{e}_{\boldsymbol{0}})$ is invertible. Since it is an invertible anti-diagonal matrix, its inverse is anti-diagonal as well. We now determine the support of $\mathcal{H}^{\mathrm{pol}}_f(\mathbf{e}_{\boldsymbol{0}})^{-1}\mathcal{H}^{\mathrm{pol}}_f(\mathbf{e}_{\boldsymbol{k}})$. If its $(i,j)$-entry is nonzero, then there exists an index $s$
such that
\begin{equation*}
\big(
\mathcal{H}^{\mathrm{pol}}_f
(\mathbf{e}_{\boldsymbol{0}})^{-1}
\big)_{is}\neq0
\qquad\text{and}\qquad
\big(
\mathcal{H}^{\mathrm{pol}}_f
(\mathbf{e}_{\boldsymbol{k}})
\big)_{sj}\neq0.
\end{equation*}
The first condition implies $i+s=d-1$, whereas \eqref{isobaric-pol-Hessian-support} gives $s+j+|\boldsymbol{k}|=d-1$. Therefore
\begin{equation}\label{isobaric-normalized-support}
j=i-|\boldsymbol{k}|.
\end{equation}
Consequently, whenever $|\boldsymbol{k}|>0$, the matrix
\begin{equation}\label{isobaric-strict-lower}
\mathcal{H}^{\mathrm{pol}}_f
(\mathbf{e}_{\boldsymbol{0}})^{-1}
\mathcal{H}^{\mathrm{pol}}_f
(\mathbf{e}_{\boldsymbol{k}})
\end{equation}
is strictly lower triangular. By Lemma~\ref{lem:Hessian} and Theorem~\ref{thm:hessian}, $f$ is persistent. Let $P_f$ be the multilinear form in \eqref{polarized-Hessian-perfect-power}, so that $\Hess(f_{\mathbf{v}})=P_f(\mathbf{v})^d$. Since $\Hess(f_{\mathbf{e}_{\boldsymbol{0}}})\neq0$, we have $P_f(\mathbf{e}_{\boldsymbol{0}})\neq0$. For every $\boldsymbol{k}$ with $|\boldsymbol{k}|>0$, \eqref{isobaric-strict-lower} is singular, while $\mathcal{H}^{\mathrm{pol}}_f(\mathbf{e}_{\boldsymbol{0}})$ is invertible. Hence $\Hess(f_{\mathbf{e}_{\boldsymbol{k}}})=0$, and therefore
\begin{equation}\label{isobaric-Pf-vanishing}
P_f(\mathbf{e}_{\boldsymbol{k}})=0,
\qquad |\boldsymbol{k}|>0.
\end{equation}
Choose $c\in\mathbbm{C}^{\times}$ such that $c\,P_f(\mathbf{e}_{\boldsymbol{0}})=1$ and take $\mathbf{v}_0=(c e_0,e_0,\ldots,e_0)$ in the definition \eqref{normalized-polarized-Hessian}. By multilinearity, $\mathcal{H}^{\mathrm{pol}}_f(\mathbf{v}_0)=c\,\mathcal{H}^{\mathrm{pol}}_f(\mathbf{e}_{\boldsymbol{0}})$, and hence $\mathcal{H}^{\mathrm{pol}}_f(\mathbf{v}_0)^{-1}=c^{-1}\mathcal{H}^{\mathrm{pol}}_f(\mathbf{e}_{\boldsymbol{0}})^{-1}$. For $|\boldsymbol{k}|>0$, using
\eqref{isobaric-Pf-vanishing}, we obtain
\begin{equation*}
\mathcal{N}^{\mathrm{pol}}_f(\mathbf{e}_{\boldsymbol{k}})=\mathcal{H}^{\mathrm{pol}}_f(\mathbf{v}_0)^{-1}
\mathcal{H}^{\mathrm{pol}}_f(\mathbf{e}_{\boldsymbol{k}})-P_f(\mathbf{e}_{\boldsymbol{k}})I_V=c^{-1}
\mathcal{H}^{\mathrm{pol}}_f(\mathbf{e}_{\boldsymbol{0}})^{-1}\mathcal{H}^{\mathrm{pol}}_f(\mathbf{e}_{\boldsymbol{k}}),
\end{equation*}
which is strictly lower triangular by \eqref{isobaric-normalized-support}. For $\boldsymbol{k}=\boldsymbol{0}$, one has
\begin{equation*}
\mathcal{N}^{\mathrm{pol}}_f(\mathbf{e}_{\boldsymbol{0}})=c^{-1}I_V-P_f(\mathbf{e}_{\boldsymbol{0}})I_V=0,
\end{equation*}
because $c^{-1}=P_f(\mathbf{e}_{\boldsymbol{0}})$. Finally, the map $\mathcal{N}^{\mathrm{pol}}_f:V^{\times r}\to \operatorname{End}(V)$ is multilinear. Since the tuples $\mathbf{e}_{\boldsymbol{k}}$ span the corresponding tensor product, every value $\mathcal{N}^{\mathrm{pol}}_f(\mathbf{v})$ is a linear combination of the matrices $\mathcal{N}^{\mathrm{pol}}_f(\mathbf{e}_{\boldsymbol{k}})$. Thus every element of $\mathscr{N}^{\mathrm{pol}}_f$ is strictly lower triangular in the standard basis. Reversing the order of the basis makes all these matrices strictly upper triangular. Therefore $\mathscr{N}^{\mathrm{pol}}_f$ is simultaneously strictly triangularizable.
\end{proof}

\begin{remark}\label{rem:isobaric-triangularizable}
Proposition~\ref{prop:isobaric-triangularizable} shows that the distinguished isobaric class considered in Section~3 is contained in the triangularizable class of the present section. Thus simultaneous strict triangularizability of normalized polarized Hessian spaces provides a matrix-theoretic explanation for the Kronecker closure of this isobaric class. In particular, the closure result of Corollary~\ref{cor:persistent-isobaric-product} also follows from Theorem~\ref{thm:general-triangularizability}. The counterexample of Section~\ref{sec:counterexample} shows that this additional structure does not follow from persistence alone.
\end{remark}

\begin{remark}
The one-sided criterion of Corollary~\ref{cor:one-sided-triangularizability} uses a feature specific to the cubic case. When $n=3$, the polarized Hessian depends linearly on a single vector, so a weighted sum of normalized Hessian matrices can again be expressed by evaluating the same linear map at one vector. For $n>3$, the polarized Hessian is multilinear in $n-2$ arguments, and the corresponding weighted sum need not arise from evaluation at a decomposable $(n-2)$-tuple. Therefore the proof of the one-sided cubic criterion does not directly extend to arbitrary degree.
\end{remark}

%%%%%%%%%%%%%%%%%%%%%%%%%%%%%%%%%%%%%%%%%%%%%%%%%%%%%%%%%%%%
%%%%%%%%%%%%%%%%%%%%%%%%%%%%%%%%%%%%%%%%%%%%%%%%%%%%%%%%%%%%
\section{Failure of Kronecker closure in general}\label{sec:counterexample}

The closure results of the previous sections rely on additional structural properties of the factors. We now show that such hypotheses cannot be removed in general. More precisely, we construct a persistent cubic form whose Kronecker square is not persistent. The obstruction occurs at a nondecomposable point of the tensor-product space, and therefore is not detected by the Hessian identity on the Segre variety established in Lemma~\ref{lem:Kronecker-Hessian-decomposable}.

\begin{proposition}
There exists a persistent cubic form $f\in \Sym^3\mathbbm{C}^{12}$ such that $f\boxtimes f\in\Sym^3(\mathbbm{C}^{12}\otimes\mathbbm{C}^{12})\simeq \Sym^3\mathbbm{C}^{144}$ is not persistent. In particular, symmetric persistence is not preserved under Kronecker products in general.
\end{proposition}
\begin{proof}
Let $x=(x_0,\ldots,x_4)$ and define the quadratic map
\begin{equation*}
Q=(q_0,\ldots,q_4):\mathbbm{C}^5\longrightarrow\mathbbm{C}^5
\end{equation*}
by $q_0(x)=\frac12x_2^2$, $q_1(x)=x_0x_2+\frac12x_2^2-x_2x_4$, $q_2(x)=x_0x_3+\frac12x_1^2-x_1x_4-x_3x_4+\frac12x_4^2$, $q_3(x)=x_0x_4-x_1x_2+x_2x_4-x_4^2$, and $q_4(x)=\frac12x_2^2$. Its Jacobian matrix is
\begin{equation}\label{Jacobian-Q}
\mathcal{J}_Q(x)=
\begin{pmatrix}
0&0&x_2&0&0\\
x_2&0&x_0+x_2-x_4&0&-x_2\\
x_3&x_1-x_4&0&x_0-x_4&-x_1-x_3+x_4\\
x_4&-x_2&-x_1+x_4&0&x_0+x_2-2x_4\\
0&0&x_2&0&0
\end{pmatrix}.
\end{equation}
A direct determinant computation gives $\det(\mu I_5-\mathcal{J}_Q(x))=\mu^5$. Hence, by the Cayley--Hamilton theorem,
\begin{equation}\label{J-nilpotent}
\mathcal{J}_Q(x)^5=0,\qquad\forall\,\,x\in\mathbbm{C}^5.
\end{equation}
Thus every matrix $\mathcal{J}_Q(x)$ is nilpotent. We now introduce coordinates $(s,x,y,t)$ on $\mathbbm{C}^{12}$, where $x=(x_0,\ldots,x_4)$, $y=(y_0,\ldots,y_4)$, and define
\begin{equation}\label{f-counterexample}
f(s,x,y,t)\coloneqq s^2t+2s\sum_{i=0}^4y_ix_i+\sum_{i=0}^4y_iq_i(x)\in\Sym^3\mathbbm{C}^{12}.
\end{equation}
For convenience, put
\begin{equation*}
S(y)\coloneqq\sum_{i=0}^4y_i\mathcal{H}_{q_i},
\qquad\qquad
A(s,x)\coloneqq 2sI_5+\mathcal{J}_Q(x),
\end{equation*}
where $\mathcal{H}_{q_i}$ denotes the constant Hessian matrix of the quadratic form $q_i$. With respect to the block decomposition $\mathbbm{C}^{12}=\mathbbm{C}_s\oplus\mathbbm{C}^5_x\oplus\mathbbm{C}^5_y\oplus\mathbbm{C}_t$, the Hessian matrix of $f$ is
\begin{equation}\label{Hf-block}
\mathcal{H}_f(s,x,y,t)=
\begin{pmatrix}
2t & 2y^{\mathsf T} & 2x^{\mathsf T} & 2s \\
2y & S(y) & A(s,x)^{\mathsf T} & 0 \\
2x & A(s,x) & 0 & 0 \\
2s & 0 & 0 & 0
\end{pmatrix}.
\end{equation}
Expanding the determinant first along the last row and then along the last column gives
\begin{equation}\label{Hf-det1}
\Hess(f)=-4s^2\det\begin{pmatrix}
S(y)&A(s,x)^{\mathsf T} \\
A(s,x)&0
\end{pmatrix}.
\end{equation}
For arbitrary $5\times5$ matrices $S$ and $A$ one has
\begin{equation}\label{block-det}
\det\begin{pmatrix}
S&A^{\mathsf T} \\
A&0
\end{pmatrix}=(-1)^5\det(A^{\mathsf T})\det(A)=-\det(A)^2.
\end{equation}
Indeed, exchanging the two block columns transforms the matrix into the
block upper-triangular matrix
\begin{equation*}
\begin{pmatrix}
A^{\mathsf T}&S \\
0&A
\end{pmatrix},
\end{equation*}
and exchanging two blocks of five columns contributes the sign $(-1)^{5\cdot5}=-1$. It follows from \eqref{Hf-det1}--\eqref{block-det} that $\Hess(f)=4s^2\det(2sI_5+\mathcal{J}_Q(x))^2$. Since $\mathcal{J}_Q(x)$ is nilpotent by \eqref{J-nilpotent}, $\det(2sI_5+\mathcal{J}_Q(x))=(2s)^5$. Therefore
\begin{equation}\label{Hf-perfect-power}
\Hess(f)=4s^2(2s)^{10}=(2s)^{12}.
\end{equation}
By Theorem~\ref{thm:hessian}, the identity \eqref{Hf-perfect-power} implies that $f$ is persistent. We now prove that $f\boxtimes f$ is not persistent. Let $b=(1,0,\ldots,0)\in\mathbbm{C}^{12}$, where the nonzero coordinate is the $s$-coordinate. From \eqref{Hf-block},
\begin{equation}\label{G_f}
\mathcal{H}_f(b)=\begin{pmatrix}
0&0&0&2\\
0&0&2I_5&0\\
0&2I_5&0&0\\
2&0&0&0
\end{pmatrix},
\end{equation}
and in particular
\begin{equation}\label{det-G_f}
\Hess(f)(b)=2^{12}\neq0.
\end{equation}
Let $\varepsilon_0,\ldots,\varepsilon_4$ denote the standard basis of $\mathbbm{C}^5_x$, and define $u_0=(0,\varepsilon_0,0,0)$ and $u_1=(0,\varepsilon_1,0,0)$. More generally, for $\xi\in\mathbbm{C}^5_x$, put $u_\xi=(0,\xi,0,0)\in\mathbbm{C}_s\oplus\mathbbm{C}^5_x\oplus\mathbbm{C}^5_y\oplus\mathbbm{C}_t$. Using \eqref{Hf-block} and \eqref{G_f}, one obtains
\begin{equation*}
\mathcal{N}(\xi)\coloneqq
\mathcal{H}_f(b)^{-1}\mathcal{H}_f(u_\xi)=
\begin{pmatrix}
0&0&0&0\\
\xi&\frac12\mathcal{J}_Q(\xi)&0&0\\
0&0&\frac12\mathcal{J}_Q(\xi)^{\mathsf T}&0\\
0&0&\xi^{\mathsf T}&0
\end{pmatrix}.
\end{equation*}
In particular, the $5$-dimensional subspace $X\coloneqq \{u_\xi\mid\xi\in\mathbbm{C}^5_x\}\subset\mathbbm{C}^{12}$ is invariant under every $\mathcal{N}(\xi)$, and
\begin{equation*}
\mathcal{N}(\xi)|_X=\frac{1}{2}\mathcal{J}_Q(\xi).
\end{equation*}
Set $\mathcal{N}_0\coloneqq\mathcal{N}(\varepsilon_0)$ and $\mathcal{N}_1\coloneqq\mathcal{N}(\varepsilon_1)$. From \eqref{Jacobian-Q},
\begin{equation}\label{E0E1}
E_0\coloneqq \mathcal{J}_Q(\varepsilon_0)=
\begin{pmatrix}
0&0&0&0&0\\
0&0&1&0&0\\
0&0&0&1&0\\
0&0&0&0&1\\
0&0&0&0&0
\end{pmatrix},
\qquad
E_1\coloneqq \mathcal{J}_Q(\varepsilon_1)=
\begin{pmatrix}
0&0&0&0&0\\
0&0&0&0&0\\
0&1&0&0&-1\\
0&0&-1&0&0\\
0&0&0&0&0
\end{pmatrix}.
\end{equation}
Hence
\begin{equation}\label{N0N1-restriction}
\mathcal{N}_0|_X=\frac12E_0,
\qquad\qquad
\mathcal{N}_1|_X=\frac12E_1.
\end{equation}
Since $\dim(\mathbbm{C}^{12}\otimes\mathbbm{C}^{12})=144$, if $f\boxtimes f$ were persistent, then by Theorem~\ref{thm:hessian} there would exist a nonzero linear form $\ell\in(\mathbbm{C}^{12}\otimes\mathbbm{C}^{12})^\vee$ and a scalar $\lambda\in\mathbbm{C}^{\times}$ such that $\Hess(f\boxtimes f)(U)=\lambda\,\ell(U)^{144}$ for every $U\in\mathbbm{C}^{12}\otimes\mathbbm{C}^{12}$. After rescaling $\ell$, we may assume $\lambda=1$, and hence
\begin{equation}\label{fKf-assume-persistent}
\Hess(f\boxtimes f)(U)=\ell(U)^{144},
\qquad
\forall\,\,U\in\mathbbm{C}^{12}\otimes\mathbbm{C}^{12}.
\end{equation}
Set $B\coloneqq b\otimes b$. For cubic forms, \eqref{Hessian-Kronecker} gives
\begin{equation}\label{H_fKf_B}
\mathcal{H}_{f\boxtimes f}(B)=\frac{1}{6}\,\mathcal{H}_f(b)\boxtimes\mathcal{H}_f(b),
\end{equation}
which is invertible by \eqref{det-G_f}. Hence, under the assumption \eqref{fKf-assume-persistent}, $\ell(B)\neq0$. On the other hand, since the $s$-coordinate of $u_0$ and $u_1$ is zero, \eqref{Hf-perfect-power} gives $\Hess(f)(u_i)=0$ for $i=0,1$. Therefore $\mathcal{H}_{f\boxtimes f}(u_i\otimes u_i)=\frac{1}{6}\,\mathcal{H}_f(u_i)\boxtimes \mathcal{H}_f(u_i)$ is singular. It follows from \eqref{fKf-assume-persistent} that
\begin{equation}\label{ell-u_i}
\ell(u_i\otimes u_i)=0,
\qquad i=0,1.
\end{equation}
Define the nondecomposable tensor $z\coloneqq u_0\otimes u_0+u_1\otimes u_1$. Since $\ell$ is linear, \eqref{ell-u_i} implies $\ell(z)=0$. Consequently, $\ell(B+cz)=\ell(B)$ for every $c\in\mathbbm{C}$, and hence \eqref{fKf-assume-persistent} gives
\begin{equation}\label{det-constant}
\Hess(f\boxtimes f)(B+cz)=\Hess(f\boxtimes f)(B).
\end{equation}
Because $f\boxtimes f$ is cubic, its Hessian matrix depends linearly on its argument. Therefore $\mathcal{H}_{f\boxtimes f}(B+cz)=\mathcal{H}_{f\boxtimes f}(B)+c\,\mathcal{H}_{f\boxtimes f}(z)$. Define
\begin{equation*}
M\coloneqq \mathcal{H}_{f\boxtimes f}(B)^{-1}\mathcal{H}_{f\boxtimes f}(z).
\end{equation*}
Dividing \eqref{det-constant} by $\Hess(f\boxtimes f)(B)\neq0$, we obtain $\det(I_{144}+cM)=1$ for every $c\in\mathbbm{C}$. It follows that all eigenvalues of $M$ are zero; equivalently, $M$ is nilpotent. We now compute $M$ explicitly enough to obtain a contradiction. By the linearity of the Hessian matrix of the cubic $f\boxtimes f$ and \eqref{H_fKf_B},
\begin{equation*}
\mathcal{H}_{f\boxtimes f}(z)=\mathcal{H}_{f\boxtimes f}(u_0\otimes u_0)+\mathcal{H}_{f\boxtimes f}(u_1\otimes u_1)=\frac{1}{6}
\big(\mathcal{H}_f(u_0)\boxtimes\mathcal{H}_f(u_0)
+\mathcal{H}_f(u_1)\boxtimes\mathcal{H}_f(u_1)\big).
\end{equation*}
Moreover, from \eqref{H_fKf_B}, $\mathcal{H}_{f\boxtimes f}(B)^{-1}=6(\mathcal{H}_f(b)^{-1}\boxtimes\mathcal{H}_f(b)^{-1})$. Hence $M=\mathcal{N}_0\boxtimes\mathcal{N}_0+\mathcal{N}_1\boxtimes\mathcal{N}_1$. Since $X$ is invariant under both $\mathcal{N}_0$ and $\mathcal{N}_1$, the space $X\otimes X$ is invariant under $M$. By
\eqref{N0N1-restriction},
\begin{equation}\label{eq:Mrestriction}
M|_{X\otimes X}=\frac{1}{4}
\left(E_0\boxtimes E_0+E_1\boxtimes E_1\right).
\end{equation}
We show that $M$ is not nilpotent. A direct computation using the explicit matrices in \eqref{E0E1} and \eqref{eq:Mrestriction} gives
\begin{equation*}
\det(\mu I_{25}-M|_{X\otimes X})=\mu^{19}\Big(\mu^2-\frac18\Big)\Big(\mu^2+\frac1{16}\Big)^2.
\end{equation*}
In particular, $M|_{X\otimes X}$ has nonzero eigenvalues and hence is not nilpotent. Since $X\otimes X$ is an invariant subspace of $M$, it follows that $M$ itself is not nilpotent. This contradicts the previous conclusion. Therefore $f\boxtimes f$ is not persistent. Since $f$ itself is persistent, this gives a counterexample to closure of symmetric persistent tensors under Kronecker products.
\end{proof}

\begin{remark}\label{rem:counterexample-nontriangularizable}
The same construction also shows that the normalized Hessian space of a
persistent cubic need not be simultaneously strictly triangularizable.
Indeed, for the persistent cubic
$f\in\Sym^3\mathbbm{C}^{12}$ constructed above, the normalized Hessian
operators $\mathcal{N}_0$ and $\mathcal{N}_1$ preserve the subspace $X$ and satisfy \eqref{N0N1-restriction}. Since $\operatorname{tr}(E_0^2E_1^2)=-1\neq0$, the matrices $E_0$ and $E_1$ cannot be simultaneously strictly triangularizable. Consequently, neither can the normalized Hessian space $\mathcal{N}_f(\ker\ell)$. Thus the triangularizability hypothesis in Section~\ref{sec:matrix-space} is a genuine additional condition and does not follow from persistence alone.
\end{remark}

It is worth noting that the dimension $12$ arising in the above construction is natural. Meisters and Olech proved that for homogeneous quadratic maps $Q:\mathbbm{C}^m\to\mathbbm{C}^m$ with nilpotent Jacobian, strong nilpotence necessarily holds when $m<5$, whereas counterexamples already occur in dimension $m=5$ \cite{MO91}. Since the above construction associates to such a quadratic map in dimension $m$ a cubic form in dimension $2m+2$, the smallest dimension accessible by this construction is therefore $12$.

\begin{remark}
The obstruction occurs at the nondecomposable tensor $z=u_0\otimes u_0+u_1\otimes u_1$. Thus it is invisible on the Segre variety. On decomposable tensors
$v\otimes w$, the identity
\begin{equation*}
\mathcal{H}_{f\boxtimes f}(v\otimes w)=
\frac{1}{6}\,\mathcal{H}_f(v)\boxtimes\mathcal{H}_f(w)
\end{equation*}
and the corresponding Kronecker determinant formula remain valid. The
failure occurs precisely when one attempts to extend the perfect-power
determinant identity from decomposable tensors to arbitrary elements of
$\mathbbm{C}^{12}\otimes\mathbbm{C}^{12}$.
\end{remark}

%%%%%%%%%%%%%%%%%%%%%%%%%%%%%%%%%%%%%%%%%%%%%%%%%%%%%%%%%%%%

\section{Acknowledgments}
The author warmly thanks Giorgio Ottaviani for carefully reading the manuscript, for his invaluable feedback and insightful comments, for pointing out the global differentiation identities underlying Proposition~\ref{prop:Kronecker-differentiation}, and for encouraging a deeper investigation of the cubic case. This work was supported by the European Union through the ERC Advanced Grant \emph{TAtypic}, Project No.~101142236. Views and opinions expressed are, however, those of the author only and do not necessarily reflect those of the European Union or the European Research Council Executive Agency. Neither the European Union nor the granting authority can be held responsible for them.

Generative AI (ChatGPT, OpenAI) was used solely as an auxiliary tool in the exploration and computational testing of the counterexample presented in Section~\ref{sec:counterexample}, under the mathematical direction of the author. The mathematical construction, proof, and interpretation, as well as all other results of the paper, were developed and verified by the author.

%%%%%%%%%%%%%%%%%%%%%%%%%%%%%%%%%%%%%%%%%%%%%%%%%%%%%%%%%%%%
%\appendix
%%%%%%%%%%%%%%%%%%%%%%%%%%%%%%%%%%%%%%%%%%%%%%%%%%%%%%%%%%%%

%%%%%%%%%%%%%%%%%%%%%%%%%%%%%%%%%%%%%%%%%%%%%%%%%%%%%%%%%%%%

%%%%%%%%%%%%%%%%%%%%%%%%%%%%%%%%%%%%%%%%%%%%%%%%%%%%%%%%%%%%
\end{document}